%% file: def-mcg.tex
\documentclass[a4paper,11pt,final]{amsart}
\usepackage[margin=7cc]{geometry}

\usepackage[textsize=footnotesize,color=yellow!50]{todonotes}

\usepackage[utf8]{inputenc}
\usepackage[T1]{fontenc}
\usepackage[english]{babel}
\usepackage{csquotes}
\usepackage{amssymb}
\usepackage{amsmath,amsthm}

\usepackage{graphicx}
\usepackage{subcaption}

\usepackage{tikz}
\usetikzlibrary{cd}
\usepackage[notref,notcite]{showkeys}
\usepackage{datetime}

\usepackage{BOONDOX-cal}

\usepackage[shortlabels]{enumitem}
\usepackage[pdfauthor={D. Schrittesser},
            pdftitle={Constructing maximal cofinitary groups},
            pdfproducer={Latex with hyperref}]{hyperref}
\include{preamble}

\newcommand{\cofperm}{S_\infty}

\DeclareMathOperator{\parinj}{pari}

\newcommand{\infinitesubsets}[1]{{#1}^{[\infty]}}

\newcommand{\fcode}{c}

\newcommand{\infG}{\mathcal C}

\newcommand{\glue}[1]{\mathbin{\sqcup_{#1}}}
\DeclareMathOperator{\codeword}{\mathcal{w}}
\DeclareMathOperator{\diffword}{\mathcal{w}}

\newcommand{\po}[1]{\mathbin{\prec_{#1}}}

\newcommand{\porange}[1]{\mathbin{\prec_{\#}}}
\newcommand{\pow}[1]{\po{#1}}

\newcommand{\adsetname}{\operatorname{D}}

\newcommand{\setname}{\operatorname{D}}
\newcommand{\set}[1]{\setname(#1)}
\newcommand{\cfname}{\chi}
\newcommand{\cf}[1]{\cfname(#1)}
\newcommand{\invmap}{\operatorname{\cfname^{-1}}}

\newcommand{\x}{x}

\providecommand{\Isat}{I}

\newcommand{\mcg}{\dot{\mathcal C}}

\DeclareMathOperator{\cofmap}{\mathcal c}
\DeclareMathOperator{\cofmapext}{\xi}
\DeclareMathOperator{\maxmap}{\dot{\cofmap}}

\newcommand{\actson}{\curvearrowright}

\newcommand{\partialto}{\rightharpoonup}
\newcommand{\gluemap}{\sqcup}

\newcommand{\Su}{E}
\newcommand{\Sud}{D^{\dagger}}

\author[Schrittesser]{David Schrittesser}
\address{Department of Mathematics, %
University of Toronto, 40 St. George Street, Toronto, Ontario, Canada M5S 2E4\\
\textit{and}
Institute for Advanced Study in Mathematics, Harbin Institute of Technology, 92 West Da Zhi Street Harbin, Heilongjiang 150001, China}

\email{david@logic.univie.ac.at}

\title{Constructing maximal cofinitary groups}



\begin{document}

\begin{abstract}
Improving and clarifying a construction of Horowitz and Shelah, we show how to construct (in $\ZF$, that is, without using the Axiom of Choice) maximal cofinitary groups.
Among the groups we construct, one is definable by a formula in second order arithmetic with only a few natural number quantifiers.
\end{abstract}

\maketitle

A \emph{cofinitary group} is a subgroup of $S_\infty$ (the group of bijections from $\nat$ to itself) each non-identity element of which leaves at most finitely many points fixed.
A \emph{maximal cofinitary group} (short: MCG) is one which is maximal among cofinitary groups with respect to $\leq$, i.e., is not a proper subgroup of a cofinitary group.

Maximal cofinitary groups were so named by Cameron.
In \cite{cameron-cofinitary,cameron-aspects} Cameron proposes the study of the class of cofinitary groups, as a ``dual'' class to the finitary groups, that is, permutation groups where every element moves only finitely many points.\footnote{Other previous work in group theory on cofinitary groups includes \cite{truss-embeddings, adeleke-embeddings}.}
While the finitary groups already possessed a well-developed structure theory, the class of cofinitary groups (which contains, for example, all Tarski monster groups) had to be much more complicated.
For example, the group of all finitary permutations is the unique maximal finitary group.
Of course every cofinitary group can be enlarged to a MCG by Zorn's lemma (a.k.a.\ the Axiom of Choice).
Already Truss and Adeleke had shown \cite{truss-embeddings, adeleke-embeddings} that no MCG can be countable.
Hjorth \cite{hjorth-cameron} showed that any closed subgroup of $S_\infty$ is the continuous homomorphic image of a closed cofinitary group (refuting a conjecture of Cameron, made  in \cite{cameron-cofinitary}, as he says, with ``some trepidation''). 

\medskip

Set theorists have long been interested in MCGs (see, e.g., \cite{koppelberg-groups}).
One long line of research regards their size (see e.g.\ \cite{zhang-maximal,zhang-cofinitary,zhang-adjoining,zhang-adjoining2,zhang-permutation,hrusak-cofinitary,brendle-uniformity,kastermans-cardinal,fischer-toernquist-template}).
Questions about MCGs on $\kappa$, where $\kappa$ is an uncountable cardinal, have also been studied by Fischer and Switzer \cite{fischer-maximal,fischer-switzer}.
The isomorphism types of MCGs have been investigated in \cite{kastermans-isomorphism}.

\medskip

The line of research to which this paper belongs concerns the \emph{definability} of MCGs.
Many objects which were first constructed using the Axiom of Choice, can be shown to be necessarily very irregular---much like the paradoxical decomposition of the sphere, which has to consist of non-measurable pieces.
Such objects then cannot have low definitional complexity---such as, being Borel.
This pattern was shown by Mathias to hold for so-called MAD families \cite{mathias-thesis,mathias}, whose definition is superficially similar to MCGs.

So a natural question for MCGs arose: Does a Borel MCG exist? Can its existence be ruled out? What is the least possible definitional complexity of an MCG? This is related to the question whether the Axiom of Choice is necessary for the construction of an MCG: By a well-known argument using Levy-Shoenfield absoluteness, if a Borel MCG can be constructed, then any use of the Axiom of Choice becomes spurious.

\medskip

Let us give a quick review, for the non-expert, of notions of definability from descriptive set theory as they are used in this article.
Some of these are of course merely topological: 
The Borel sets are stratified into a hierarchy, with the open and closed sets at the bottom, followed by the $F_\sigma$ (countable unions of closed) sets and the $G_\delta$ (countable intersections of open) sets.
Open, closed, $F_\sigma$, $G_\delta$, are also denoted by $\mathbf{\Sigma}^0_1$, $\mathbf{\Pi}^0_1$, $\mathbf{\Sigma}^0_2$, and $\mathbf{\Pi}^0_2$ sets, respectively.
Similarly, $\mathbf{\Sigma}^0_3$ denotes $G_{\delta\sigma}$, etc.; $\mathbf{\Sigma}^0_{<\omega}$ denotes the finite level Borel sets.

Beyond the Borel sets, we speak of analytic sets (continuous images of ${}^{\nat}\nat$, or equivalently, projections of closed sets) denoted by $\mathbf{\Sigma}^1_1$ and their complements, the co-analytic sets or $\mathbf{\Pi}^1_1$ sets. It is a classic fact that the Borel sets are precisely the sets in $\mathbf{\Delta}^1_1 := \mathbf{\Sigma}^1_1 \cap \mathbf{\Pi}^1_1$.

Finally, all these complexity classes have ``lightface'' (also called ``effective'') counterparts.
In what follows, the reader will not loose much if they ignore the distinction and replace the ``lightface'' classes by their ``boldface'' counterparts (which we have just described) everywhere.

For those interested, let me illustrate the distinction quickly by example: 
E.g., $\Sigma^0_1$ is the collection of ``effectively open'' or ``computably open'' sets, 
that is, unions of basic open neighborhoods, where the neighborhoods making up the union are listed (or, their codes are listed) by a computable function.
Likewise, the function enumerating the effectively open sets (better: their codes) in the intersection forming a $\Pi^0_1$ (or ``computably $G_\delta$'') set is required to be computable.

It is a basic fact of descriptive set theory that the complexity of a set can be bounded from above by counting quantifiers in (one of) its definition(s); e.g., $\Sigma^0_n$ sets are defined by formulas with at most $n$ changes of quantifiers over natural numbers, starting with ``$\exists$'', resp.\ starting with $\forall$ in the case of $\Pi^0_n$. The same holds for $\Sigma^1_n$ and $\Pi^1_n$ where one counts quantifiers over ${}^{\nat}\nat$ instead.

Moreover, the boldface classes arise from holding a parameter fixed: If $\{(x,y) \setdef P(x,y)\}$ is $\Pi^1_n$ (say), then given any $x$, 
$\{y \setdef P(x,y)\}$ is $\mathbf{\Pi}^1_n$, and every $\mathbf{\Pi}^1_n$ set arises in this way. 
The same holds for all the $\Sigma$ and $\Pi$ classes mentioned above.
Therefore, since the defining formulas of the sets in this article are parameter-free, it is simply more precise to state the complexity in terms of the lightface hierarchy.

For a deeper introduction, and as a general reference for descriptive set theory, we recommend \cite{kechris}, \cite{moschovakis}, and \cite{mansfield}.

\medskip

We can now continue our short history of definability of MCGs.
Kastermans showed in \cite[Theorem 10]{kastermans-complexity} that no MCG can be contained in a $K_\sigma$ set, i.e., in a countable union of sets which are compact.
Gao and Zhang \cite{gao-definable} showed that on the other hand, assuming the Axiom of Constructibility, there is a MCG with a co-analytic (in fact, $\Pi^1_1$) generating set.
This was improved by Kastermans' theorem \cite{kastermans-complexity} 
that under the Axiom of Constructibility,  there is a co-analytic (in fact, $\Pi^1_1$) MCG.

In 2016, just after Vera Fischer, Asger Törnquist and the present author had constructed a $\Pi^1_1$ MCG 
in the constructible universe which (has size $\omega_1$ but) remains maximal after adding Cohen reals \cite{MCG.coanalytic},  Horowitz and Shelah \cite{mcg-borel} gave a construction of a MCG without using the Axiom of Choice or any similar choice principle. Not only did they work in (choice-less) Zermelo-Fraenkel set theory (\ZF),
moreover, their construction yields a Borel MCG (it would be enough to present an analytic such group; by the maximality property of such groups, being analytic implies being Borel).

\medskip

In this article we present a simpler construction of definable MCGs in $\ZF$. 
This construction takes some important ideas from the earlier work of Horowitz and Shelah, but also differs substantially in places; similarities and differences are discussed below in Remark \ref{r.comparison}. 

In fact, the present paper describes more than one such constructions.
The first is a construction of a MCG in $\ZF$, based on a combinatorial sufficient condition for cofinitariness and maximality (Proposition~\ref{p.bp}).  
Secondly, we show how to alter the construction (using the same sufficient condition) to obtain an MCG whose definitional complexity is low: 
Namely first, an MCG which is Borel, and then, with just a little more attention to detail and a tiny change in the construction, one which is arithmetical, i.e., can be defined in second order arithmetic by a formula which uses only quantifiers over natural numbers.

\begin{itheorem}\label{t.arithmetic}
There is a maximal cofinitary group which is finite level Borel; in fact, it is definable by a $\Sigma^0_n$ formula for some $n \in \nat$, i.e., by an arithmetical formula (one involving only quantifiers over natural numbers).
\end{itheorem}
This leaves open the question of what is the optimal (i.e., lowest possible) definitional complexity of a MCG.
In particular, the following obvious question remains open: Does there exist a closed, or even an effectively closed (i.e., $\Pi^0_1$) maximal cofinitary group? 
I do not know the answer. 
A closed MCG would have a genuine claim to being obtained by concrete computation (precisely, as the complement of an effectively open set) which would be quite surprising for this type of object, defined, as it is, through a maximality condition.
Note here that there do indeed exists maximal eventually different families (``MCGs without group structure'')
which are closed, and in certain spaces, even ones which are compact \cite{schrittesser-compactness}.
The current best result is that of Kastermans \cite{kastermans-complexity} that no MCG can be contained in a $K_\sigma$ set.
The methods in this paper can, with a some effort, be pushed to yield a $\Sigma^0_2$ MCG.

\medskip

\subsection*{Some notation}

We write ${}^A B$ for the set of functions from $A$ to $B$.
Likewise, write ${}^{\nat}A$ when $A \in \{\nat, 2\}$ for Baire space resp.\ Cantor space,
and ${}^{\omega}A$, resp.\ ${}^{<\omega}A$ for the set of infinite, resp.\ finite sequences from $A$.
We use $\infinitesubsets{X}$ for the set of infinite subsets of $X$,
$S(X)$ for the group of permutations of $X$ (bijections from $X$ to $X$) and $S_\infty$ for $S(\nat)$. 
This group carries a (unique) Polish topology, but our statements about complexity of sets
refer to ${}^{\nat}\nat$.

We shall have opportunity to work with intervals in $\Int/l\Int$, the integers modulo $l$,
which are defined as follows: Given $a, b \in \Int$ (or equivalently, $a, b \in \Int/l\Int$) let
\[
[a,b] = \{\overline{a+k} \setdef k \in \nat, 0 \leq k \leq k'\text{ for the least $k' \in \nat$ s.t.\ }a+k' \equiv b \;(\bmod{l})\} 
\]
We will later work with a sequence $\vec I = (I_n)_{n\in\nat}$ of intervals in $\nat$ which form a partition of $\nat$.
We will write $\Isat(M)$ for the saturation of $M\subseteq \nat$ with respect to $\vec I$,
\[
\Isat(M) := \bigcup \{I_n \setdef n \in \nat, I_n \cap M \neq \emptyset\}.
\]

We identify $n \in \nat$ with $\{k \in \nat \setdef k< n\}$ as it allows us to use notation such as $(\forall k \in \nat\setminus n)$ for the longer ``$(\forall k \in \nat)$ if $k \geq n$ then \dots''

\subsection*{Acknowledgments}
The author would like to thank the Austrian Science Fund (FWF) for the generous support
through START Grant Y1012-N35 and thank for the support of the Government of Canada’s New Frontiers in Research Fund (NFRF), through project grant number 
NFRFE-2018-02164.
I thank Severin Mejak for his thorough reading of an earlier version, and for finding a gap in said version.

\medskip

\section{An Axiom-of-Choice-free recipe for maximal cofinitary groups}\label{s.bp}

In this section, I will give a construction in $\ZF$ (i.e., without using the Axiom of Choice) of a group $\mcg$ and then show it to be maximal cofinitary.
In fact, I will give sufficient conditions for when similar constructions yield a cofinitary and maximal cofinitary group, 
which will be useful when in the following section, a MCG of lower definitional complexity is constructed.

\medskip

I first sketch the rough, overall idea of the construction(s).
In \cite{medf-closed}, building on work of Horowitz and Shelah on  \emph{maximal eventually different families} in \cite{medf-borel}, I gave a simple recipe for constructing such a family (and the reader may find it useful to take a look at the much simpler argument in \cite{medf-borel}). 
In the following, I shall follow a similar strategy to construct a MCG.
Here are the main ideas. 

\begin{enumerate}[label={[S\arabic*]}]
\item\label{S.groundwork} Construct a perfect subset of $S_\infty$ which \emph{freely} generates a cofinitary subgroup $\infG$ of $S_\infty$.
This allows us to associate  (by a continuous map) to any $f\in \cofperm$ a generator $\xi(f)$ of $\infG$. 
The map $\xi$ is emphatically \emph{not} a homomorphism; rather, one should think of $\xi(f)$ as \emph{coding} $f$. 
We do demand additional properties of $\infG$, most notably, the orbits of $\infG$ are finite but the sequence of cardinalities of orbits grows sufficiently quickly. This additional property is needed for %
\ref{s.combinatorics} below.
\item\label{s.surgery} We describe a way to alter each $\xi(f)$ to agree with $f$ itself on an infinite set $D$, obtaining a new permutation without fixed points, denoted by
\[
\xi(f)\glue{D} f \in \cofperm.
\]
We call this ternary operation (with inputs $\xi(f)$, $D$, and $f$) surgery;
the argument $D$, i.e., the set where this new permutation  agrees with $f$, is called the \emph{transmutation site}.
Surgery straightforwardly ``merges'' two permutations, or even a permutation and a partial injective function, obtaining a permutation without fixed points under some weak assumptions on its inputs. We will have to change $\xi(f)$ not only on $D$, but on a slightly larger set $E = D \cup \Sud$, to make sure $\xi(f)\glue{D} f$ is a permutation.
\end{enumerate}
Note now that the following set
\[
 \big\la\{\xi(f) \setdef f \in S_\infty\}\big\ra^{S_\infty} 
\]
is a cofinitary group by construction.
In contrast the following set:
\begin{equation}\label{e.naive}
\{\xi(f)\glue{\set{f}} f \setdef f \in S_\infty\}
\end{equation}
where $\set{f} \in \infinitesubsets{\nat}$ is arbitrary, satisfies a maximality condition: Every element of $S_\infty$ agrees on an infinite set with a permutation from \eqref{e.naive}---with any na\"ive choice of the transmutation site $\set{f}$ for each $f\in S_\infty$; but the set in \eqref{e.naive} should not be expected to generate a cofinitary group---unless we refine our choice of $\set f$. 
The way forward is to analyze how the set in \eqref{e.naive} fails to generate a cofinitary group.
\begin{enumerate}[label={[S\arabic*]},start=3]
\item\label{s.combinatorics} By carefully choosing transmutation sites $\set{f}$ from an almost disjoint family and using the size condition from \ref{S.groundwork} on the orbits of $\infG$,  it can be arranged that the only obstacles to cofinitariness are permutations $f\in \cofperm$ 
which agree with an element of $\infG$ on an infinite subset of $\set f$.
But by this very property, we can forgo surgery for such $f$ entirely (one does have to include $\xi(f)$ as well as other elements of $\infG$ in our MCG, to achieve maximality; and one must check that not only does $f$ agree with an element of $\infG$ on an infinite set, but that this remains true after applying surgery to the generators of said element. 
Here again it is used that the sets $\set f$ are almost disjoint for different $f$, as well as a property which we call ``cooperative'', see Remark~\ref{r.superspaced} below). 
\end{enumerate}
Thus, with a careful choice of $f \mapsto \set{f}$, it becomes possible to show that the following set
\begin{equation}\label{e.first.mention}
\mcg_0 :=\{\xi(f)\glue{\set{f}} f \setdef f \in S_\infty\land \neg\kappa_{\setname}(f)\}\cup\{c\in\infG \setdef \neg(\exists f \in S_\infty)\: \xi(f) = c \land\neg\kappa_{\setname}(f)\}
\end{equation}
generates an MCG in $S_\infty$, where $\kappa_{\setname}(f)$ stands for ``$f$ agrees with an element of $\infG$ on an infinite subset of $\set{f}$'' (short: ``$f$ is caught'').
Of course, the point is that we do \emph{not} use the Axiom of Choice in choosing $\set f$ for each $f$.
The most difficult part of the proof is the analysis of how \eqref{e.naive} fails to be cofinitary; this analysis is implicit in the proof of Proposition \ref{p.cof.gen} in Section \ref{s.cof}.
After one has developed tools to deal with cofinitarity, maximality and being ``cooperative'' can be arranged with the same set of tools.

\begin{remark}\label{r.simple}
In order to obtain a group which in addition is definable by a simple formula, the idea suggests itself to refine the above strategy as follows: Instead of considering elements of $c \in \infG$ as potential codes for a permutation $f$, interpret $c$ as coding more information (and then, as before, potentially use surgery on $c$ according to this coded information). 
But the group $\infG$ which we construct below will be $K_\sigma$, i.e., a countable union of compact sets. 
Therefore it is not obvious how to use this type of approach to lower the complexity below, say, a group with a $\Pi^0_2$ set of generators (presumably, the group itself would then be $\Sigma^0_3$). 
Neither is there an obvious way to replace the group $\infG$ in the following construction by a sufficiently large (non-$K_\sigma$) cofinitary group to circumvent this problem.
It is nevertheless possible, using the methods in this paper and some additional ideas, to construct a $\Sigma^0_2$ MCG.
See also Theorem~\ref{t.F.sigma} and Question~\ref{q} below.
\end{remark}

\subsection{Ground-work: An action of the free group with a continuum of generators}\label{s.groundwork}

Our first goal is to define a group isomorphism
\[
\cofmap\colon \mathbb{F}\left({}^{\nat} 2\right) \to \infG \leq S_\infty,
\]
or equivalently, a faithful action of $\mathbb{F}\left({}^{\nat} 2\right)$ on $\nat$. 
We would like the orbits of this action to be finite, and arranged in a sequence such that their sizes exhibit sufficiently fast growth.

This action will be constructed by ``finite approximations''.
To this end, given $\alpha \leq \omega$ (i.e, $\alpha \in \nat$ or $\alpha = \nat$)
let us write 
\[
\mathbb{F}({}^\alpha 2) 
\]
for the free group with generating set ${}^\alpha2$, the set of sequences of length $\alpha$ from $\{0,1\}$, and 
for $n\in \nat$ with $n < \alpha$ 
write
\[
r^\alpha_n\colon \mathbb{F}({}^\alpha2) \to\mathbb{F}({}^n2) 
\]
for the group homomorphism defined on each generator $\x \in {}^\alpha 2$ by 
\[
r^\alpha_n(\x)=\x\res n.
\]
We can also drop the superscript since it is determined as the unique $\alpha$ such that $x \in \mathbb{F}({}^\alpha2)$; that is, we let
\[
r_n = \bigcup_{n\leq\alpha \leq \omega} r^\alpha_n.
\]
We first construct a sequence of finite groups 
\[
\la G_n \setdef n \in \nat\ra 
\]
and group homomorphisms
\begin{gather*}
c_n\colon \mathbb{F}({}^n2) \to G_n,
\end{gather*}
together with actions
\begin{multline}\label{e.I_n.basic}
\sigma_n\colon G_n \actson I_n, 
\text{acting faithfully and transitively, where $I_n = [m_n, m_{n+1})$ and }\\
\text{%
$\langle m_i \setdef i \in \nat\rangle$ is a strictly increasing sequence from $\nat$ with  $m_0=0$.}
\end{multline}
In what follows, for $n\in \nat$ let us write
\begin{equation*}\label{d.W_n}
\begin{minipage}{0.8\textwidth}
$W_n :=$ the set of (reduced) words from $\mathbb{F}({}^n2)$ of length at most $n$.
\end{minipage}
\end{equation*}
E.g., $W_0$ is the subset of the trivial group containing only the neutral element, which we take to be the empty word $\emptyset$; that is, $W_0$ is the entire group in this special case, $W_0 = \mathbb{F}({}^\emptyset 2) = \{\emptyset\}$. To give another example, $W_1 = \{\emptyset, \la 0\ra,\la 0\ra^{-1},\la 1\ra,\la 1\ra^{-1}\}$; of course $\mathbb{F}({}^1 2)$ is the free group with two generators.

Our construction of $\la G_n \setdef n \in \nat\ra$ and $c_n$ ensures the following two requirements:
For all $n\in\nat$, 
\begin{enumerate}[label=\textup{(\Alph*)}]
\item\label{r.I_n} $\sum_{m<n} \lvert I_m \rvert < \lvert I_n\vert -1$,
\item\label{r.W_n} $c_n \res W_n$ is injective. 
\end{enumerate}

\begin{proposition}\label{p.groundwork}
We can find groups $\la G_n \setdef n\in\nat\ra$, homomorphisms $\la c_n \setdef n\in \nat\ra$, and %
 actions $\sigma_n\colon G_n \actson I_n$  satisfying the above assumptions, i.e., so that  \eqref{e.I_n.basic}, \ref{r.I_n}, and \ref{r.W_n} hold.
\end{proposition}
\begin{proof}
The construction is by induction on $n$. 
Suppose we already have $G_n$ and $\sigma_n$.

Let $\la w_i \setdef i < l\ra$ be an enumeration of $W_{n+1}$ %
so that $w_0 = \emptyset$, the neutral element of $\mathbb{F}({}^{n+1}2)$.
For each $\x \in {}^{n+1}2$ let us first define a partial injection $c_0(\x)$ on $\{0, \hdots, l-1\}$
by stipulating that for any pair $i,j < l$
\[
c_0(\x) (i) = j \iff w_j = \x w_i.
\]
Now arbitrarily extend $c_0(\x)$ to a permutation $c(\x)$ of $\{0,\hdots, l-1\}$.
Let 
\[
G:=\text{ the group generated by $\{c(\x) \setdef \x \in {}^{n+1}2\}$ in $S_{l}$.}
\]
Then $c$ uniquely extends to a group homomorphism from $\mathbb{F}({}^{n+1}2)$ to $G$, which we also denote by $c$.
It is easy to see that $c$ is injective on $W_{n+1}$, as $c(w_i)(0)=i$ for each $i < l$.

Now fix some large number $k \in \nat$ and let
\begin{gather*}
G_{n+1} :=  G \times S_k,\\
c_{n+1} :=  c \times h_1,
\end{gather*}
where $h_1$ is the trivial homomorphism sending $\x$ to the identity in $S_k$.
This last part of the product is included to ensure $G_{n+1}$ is large, with the goal of establishing \ref{r.I_n}.

It is now easy to find $\sigma_{n+1}$ and $I_{n+1}$: 
Take a bijection $\iota$ of $G_{n+1}$ with an appropriate interval $I_{n+1}$ of natural numbers and let $\sigma_{n+1}$ come from the left-multiplication action of $G_{n+1}$ on itself, identified with 
$I_{n+1}$ via $\iota$.
Since 
\[
\lvert I_{n+1} \rvert = \lvert G_{n+1} \rvert \geq \lvert G \rvert \cdot k! 
\]
and $k$ can always be chosen large enough to ensure \ref{r.I_n}, we are done.
\end{proof}

\medskip

Having constructed this sequence of groups,
and actions, now define a group homomorphism 
\[
\cofmap\colon \mathbb{F}\left({}^{\nat} 2\right) \to S_\infty
\]
by describing how each generator $\x\in{}^{\nat} 2$ acts on $\nat$: 
For each $n\in \nat$, let
\begin{equation}\label{e.def}
\cofmap(\x)\res I_n = \sigma_n \circ c_n(x \res n).
\end{equation}
We now define
\begin{gather*}
\infG_0 := \cofmap\left[{}^{\nat} 2\right], \\
\infG := \cofmap\left[\mathbb{F}\left({}^{\nat} 2\right)\right] = \big\la \infG_0\big\ra^{S_\infty}.
\end{gather*}
\begin{proposition}%
The map $\cofmap$ is an injective group homomorphism and $\infG$ is a cofinitary group.
\end{proposition}

\begin{proof}
To verify injectivity, let two words $w, w' \in \mathbb{F}\left({}^{\nat} 2\right)$ be given and take $n\in\nat$ so that $w$ and $w'$ have word-length at most $n$, i.e., $\{r^\infty_n(w), r^\infty_n(w') \}\subseteq W_n$, and so that $r^\infty_n(w) \neq r^\infty_n(w')$.
Then by \ref{r.W_n}, $(c_n \circ r^\infty_n)(w) \neq (c_n \circ r^\infty_n)(w')$ and so by \eqref{e.def} also $\cofmap(w) \neq \cofmap(w')$.
Similarly, $\cofmap(w)$ is trivial or has finitely many fixed points, for any word $w \in \mathbb{F}\left({}^{\nat} 2\right)$:
Find $n\in\nat$ so that $r^\infty_n(w)\in W_n$ and $r^\infty_n(w) \neq \emptyset$ (supposing, to avoid trivialities, that $w \neq \emptyset$).
Then for each $m \geq n$, $r^\infty_{m}(w) \neq \emptyset$ and so $(\sigma_m \circ c_m \circ r^\infty_m)(w)$ has no fixed points. 
Since 
\[
\cofmap(w) \res I_m = (\sigma_m \circ c_n \circ r^\infty_m)(w),
\]  
we infer $\fix(\cofmap(w)) \subseteq \bigcup_{n'<n} I_{n'}$.
\end{proof}
It will be important to know the degree of definability of the objects constructed in this section.
The following is clear by construction.
\begin{proposition}%
The sequences $\la G_n \setdef n\in\nat\ra$, $\la I_n \setdef n\in\nat\ra$, $\la c_n \setdef n\in\nat\ra$, $\la \sigma_n \setdef n \in \nat\ra$ are each computable, i.e., $\Delta^0_1$.
Moreover $\infG_0$ is a closed subset of ${}^{\nat}\nat$ and (the graph of) $\cofmap\res{}^{\nat}2$ is closed in ${}^{\nat}2 \times {}^{\nat}\nat$. 
In fact both are $\Pi^0_1$. %
\end{proposition}

From now on, let us identify $G_n$ with a subgroup of $S(I_n)$ via $\sigma_n$.
That is, from now on we have 
\begin{gather*}
G_n \leq S(I_n),\\
c_n \colon \mathbb{F}\big({}^n 2\big) \to S(I_n),\\
\cofmap(w) \res I_n = (c_n \circ r^\infty_n)(w).
\end{gather*}
Thus, we can replace $\sigma_n$ by the action by evaluation.

Finally, given $M \subseteq \nat$ we use the notation %
\[
\Isat(M) := \bigcup \{I_n \setdef n \in \nat, I_n \cap M \neq \emptyset\}
\]
for the \emph{saturation} of a set $M$ with respect to the partition $\vec I = (I_n)_{n\in\nat}$. %

\subsection{Surgery}\label{ss.surgery}

\medskip

Write $\parinj(\nat,\nat)$ for the set of partial injective functions from $\nat$ to $\nat$. %
Largely for aesthetic reasons, let us make the following definition slightly more general than is presently needed---namely, for $f\in \parinj(\nat,\nat)$ and not just $f\in S_\infty$. 

\begin{figure}[h]
\begin{center}
\begin{tikzcd}
\dots & \arrow[l, "g"] \arrow[l, red, bend right, dashed]
(g\circ f)(n)    &  \arrow[l, "g"]    f(n) \arrow[dl, red, dashed] &  \arrow[l, "g"] \arrow[ll, red, bend right, dashed, "g \glue{D} f"] (g^{-1}\circ f)(n) & 
\arrow[l, "g"] \arrow[l, red, bend right, dashed] \dots \\
\dots & \arrow[l, "g"] \arrow[l, red, bend right, dashed]
g(n)   &  \arrow[l, "g"] \arrow[u, blue, "f"]  \arrow[u, red, bend right, dashed]   n  \in D  &  \arrow[l, "g"]  \arrow[l, red, bend right, dashed] g^{-1}(n) & 
\arrow[l, "g"] \arrow[l, red, bend right, dashed] \dots
\end{tikzcd}
\end{center}
\caption{Surgically transplanting $f(n)$.}\label{fig.surgery}
\end{figure}
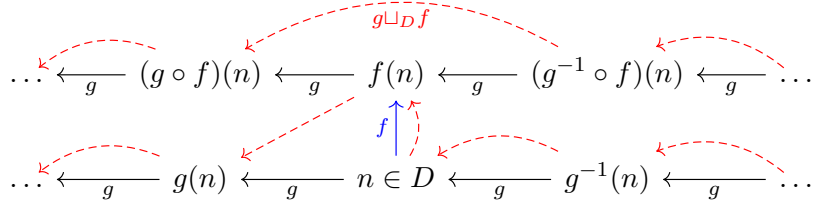

We define a partial map 
\begin{align*}
\gluemap\colon S_\infty \times \powerset(\nat) \times \parinj(\nat,\nat) &\partialto {}^{\nat}\nat,\\
(g,D,f) &\mapsto g \glue{D} f
\end{align*}
as follows:
Given $f\in \parinj(\nat,\nat)$, $D\subseteq \dom(f)$ and $g \in S_\infty$, we want to define
\[
(g \glue{D} f)\colon\nat\to\nat.
\]
If $m \in D$ and $f(m) = g(m)$, we let 
\[
(g \glue{D} f) (m) = g(m)
\]
and otherwise, writing 
\[
C = \nat\setminus\big(D \cup f[D] \cup (g^{-1}\circ f) [D]\big)
\] 
we want to let 
\begin{equation}\label{e.def.glue}
(g \glue{D} f) (m) := \begin{cases} g(m) &\text{$m\in C$,}\\
						 f(m)  &\text{ if $m\in D$,}\\
						 (g\circ f^{-1}) (m) &\text{ if $m\in f[D]$,}\\
						(g\circ g) (m) &\text{ if $m\in (g^{-1}\circ f)[D]$.}\\

\end{cases}
\end{equation}
We call this operation \emph{surgery}: $f$ is surgically grafted onto $g$ along the set $D$. 
Moreover, we shall later find it useful to use the following notation for the sets where surgery is performed:
\begin{equation}\label{e.SuSud}
\begin{alignedat}{2}
\Sud(g, D, f) &:=  f[D] \cup (g^{-1} \circ f) [D] &&= \nat\setminus (D \cup C), \\
\Su(g, D, f) &:= D \cup \Sud(g, D, f) &&= \nat\setminus C.
\end{alignedat}
\end{equation}
We now specify the domain of this operation:
For one thing, we will only consider this operation for triples $(g,D,f)$ which have the following property, which ensures that $g \glue{D} f$ on the left of \eqref{e.def.glue} is well-defined.

Let us say that $D\subseteq \nat$ is $(g,f)$-spaced if and only if
\begin{enumerate}[label=(\alph*)]
\item $D\subseteq \dom(f)$, and
\item\label{i.avoid.1} for any $m,m' \in D$ and for any 
\begin{equation*}%
h \in \{f, g^{-1} \circ f, f^{-1} \circ g^{-1} \circ f, f^{-1} \circ g \circ f \}
\end{equation*}
 it holds that $h(m) \neq m'$.
\end{enumerate}
It is not hard to see that for $(g,f)$-spaced $D$, $g \glue{D} f$  is well-defined by \eqref{e.def.glue}.
In fact, including $h = f^{-1} \circ g \circ f$ in \ref{i.avoid.1} is not needed for this; we include it for the proof that $g\glue{D} f$ is injective, below.
We let 
\[
\dom(\gluemap):=\{(g,D,f) \in S_\infty \times \powerset(\nat) \times \parinj(\nat,\nat) \setdef \id_{\nat} \notin \{g,f\} \text{ and $D$ is $(g,f)$-spaced}\}.
\]

\begin{remark}
It may hep the reader to verify that $g \glue{D} f$ can be decomposed into cycles and that these cycles are exactly the cycles of $g$ with the following modification: For each $n\in D$, if $f(n)$ and $n$ belong to different $g$-orbits, $f(n)$ is removed from whatever $g$-orbit it belongs to and inserted into the $g$-orbit of $n$ just after $n$, as shown in Figure~\ref{fig.surgery}. 
If $n$ and $f(n)$ should occur in the same $g$-orbit but $f(n) \neq n$ and $f(n)\neq g(n)$, then $f(n)$ is removed from its position, the $g$-cycle altered to lead from the predecessor of $f(n)$ to its successor immediately, and $f(n)$ is inserted in the position after $n$. 
In particular, the map $g \glue{D} f$ is a permutation of $\nat$.
\end{remark}

For the incredulous reader, we give a proof of this last fact.

\begin{lemma}\label{l.permutation}
Suppose $(g,D,f)\in \dom(\gluemap)$ (whence $D$ is $(g,f)$-spaced). 
Then 
$g \glue{D} f$ is a permutation of $\nat$, and its fixed points are precisely those of $g$.
\end{lemma}
\begin{proof}
First, we show $g \glue{D} f$ is injective.
Suppose $m, m' \in \nat$, $m \neq m'$, and 
\begin{equation}\label{e.not.inj}
g \glue{D} f (m) = g \glue{D} f (m').
\end{equation}
We omit trivial cases where by definition of $g \glue{D} f$, the above reduces to $f(m)=f(m')$ or $g(m)=g(m')$. 
By symmetry, the following three cases remain to be considered.

Firstly, suppose $m \in D$ and $m' \in f[D]$. 
Substituting the definition of $g \glue{D} f$  in \eqref{e.not.inj}, we almost immediately find
\[
m = (f^{-1} \circ g) (m'')
\]
for some $m'' \in D$ (namely, take $m''= f^{-1}(m')$).
But this is ruled out by \ref{i.avoid.1} above, i.e., by our assumption that $D$ is $(g,f)$-spaced.

The remaining two cases are similar: If $m \in D$ and $m' \in (g^{-1} \circ  f) [D]$, an analogous route as in the previous case leads us to find $m'' \in D$ such that
\[
m = (f^{-1} \circ g \circ f) (m''),
\]
and if $m \in f[D]$ and $m' \in (g^{-1} \circ  f) [D]$, we likewise obtain $m'', m'''\in D$ such that 
\[
m'' = (f^{-1} \circ g \circ f) (m''').
\]
Either contradicts \ref{i.avoid.1} above, that is, that $D$ was assumed to be $(g,f)$-spaced.

To show that $g \glue{D} f$ is surjective, let $m \in \nat$ be given and let $m' := g^{-1}(m)$.
If $m' \in C$, then $m = g(m') = g \glue{D} f (m')$ by definition.
If $m' \in  D$, $m =  g \glue{D} f (m'') = (g \circ f^{-1}) (m'')$ where $m'' = f(m')$.
If $m' \in f[D]$, $m = g \glue{D} f (m'') = g^2 (m'')$ where $m'' = g^{-1}(m')$.
Finally, if $m' \in (g^{-1} \circ f)[D]$, $m \in f[D]$, so 
$m = g \glue{D} f (m'') = f(m'')$ for $m''=f^{-1}(m)$.

The final statement regarding fixed points is obvious from the definitions. %
\end{proof}
The reader may find it helpful to note at this point that moreover, under the right circumstances, surgery does not destroy being cofinitary.
Readers can skip the following (somewhat artificial) lemma and proof sketch without loss if they wish, since we shall prove a more pertinent (but also much more complex) statement later in Proposition~\ref{p.bp}. %
\begin{lemma}
If $(g,D,f)\in \dom(\gluemap)$, $f\in S_\infty$, and $\{g, f\}$ freely generates a cofinitary group, then $g \glue{D} f$ has only finitely many fixed points.
In fact, if $\infG \cup\{f\}$ freely generates a cofinitary group, $g \in \infG$, $f\notin \infG$, and $(g,D,f) \in \dom(\gluemap)$ then  
$\{g \glue{D} f\}\cup \infG\setminus\{g\}$ generates a cofinitary group as well.
\end{lemma}

\begin{proof}[Proof sketch]
For the first assertion, by assumption, any word in the generators $f$ and $g$ has only finitely many fixed points.
Let $h:= g \glue{D} f$ and $F:= \fix(h)$; we show $F$ is finite.
This is because 
\begin{gather*}
F\cap D\subseteq \fix(f)\\ 
F\cap f[D]\subseteq \fix(g\circ f^{-1})\\
F\cap (g^{-1}\circ f)[D]\subseteq \fix(g^2) \text{, and}\\
F \setminus \Su(g,D,f) \subseteq \fix(g)
\end{gather*} 
are each finite. The second statement is left as an exercise.
\end{proof}

\subsection{The scenic route to maximality}\label{s.scenic}

Given $f \in S_\infty$ and $X \in \infinitesubsets{\nat}$, let us say \emph{$f$ is caught (by $\infG$) on $X$} to mean that for some $Y \in \infinitesubsets{X}$ and some $c \in \infG$, $f \res Y = c \res Y$.
Let us abbreviate this by $\kappa(X,f)$, i.e.,
\begin{equation}\label{e.kappa}
\kappa(X,f) :\iff \left(\exists w \in \mathbb{F}\left({}^{\nat}2\right)\right)\left(\exists Y \in \infinitesubsets X\right)\; f \res Y = \cofmap(w) \res Y.
\end{equation}

Fix a continuous one-to-one map,
\begin{equation}\label{e.chi}
\begin{gathered}
\cfname\colon \cofperm \stackrel{1-1}{\longrightarrow} {}^{\nat} 2,\\
f \mapsto \cf{f},
\end{gathered}
\end{equation}
e.g., by taking $\cf{f}$ to represent the graph of $f$ as an element of ${}^{\nat} 2$ via the obvious identification ${}^{\nat} 2 \cong {}^{\nat \times \nat}2 \cong \powerset(\nat\times\nat)$.

We thus obtain a continuous injective map $\xi$ from $\cofperm$ into $\infG$ (emphatically not a group homomorphism, nor do we need it to be onto) defined as follows:
\[
\cofmapext :=   \cofmap \circ \cfname.
\]
In the next section, we will define an injective map 
\begin{gather*}
\setname\colon \cofperm \to \infinitesubsets{\nat},\\
f \mapsto \set f
\end{gather*}
whose range will be an almost disjoint family.
This map will be defined so as to ensure that the 
following set $\mcg_0$  generates (in $S_\infty$) a maximal cofinitary group (as sketched in \ref{s.combinatorics} above):
\begin{multline}\label{e.mcg.bp.C_0}
\mcg_0 :=\big( \infG_0 \setminus \ran(\xi)\big) \cup
 \left\{\cofmapext(f) \setdef  f \in \cofperm \land\kappa(\set f,f)\right\} 
\cup \\
\left\{\cofmapext(f)\glue{\set f} f \setdef f \in \cofperm \land \neg\kappa(\set f,f)\right\}
\end{multline}
Supposing we have fixed the map $\setname$, let us introduce the following shorthands: 
\[
\kappa_{\setname}(f) :\iff \kappa(\set f, f).
\]
With this notation, the definition in \eqref{e.mcg.bp.C_0} above is obviously equivalent to the one already mentioned in \eqref{e.first.mention}.
It will be extremely convenient for what follows to introduce yet another way of referring to the elements of $\mcg_0$.
Define 
\[
\maxmap\colon {}^{\nat} 2 \to \mcg_0
\] 
as follows:
Given $\x \in {}^{\nat} 2$, 
let 
\begin{equation}\label{e.maxmap}
\maxmap(\x) :=\begin{cases} \cofmap(\x) 
					    &\text{if $\x\notin \ran(\cfname)$ or $\kappa_{\setname}\big(\cfname^{-1}(\x)\big)$},
					    \\
					     \cofmap(\x)\glue{\set{f}} f  
					     & \text{otherwise, where $f := \chi^{-1}(x)$.}				     
                        \end{cases}
\end{equation}
noting that thereby 
\begin{equation}\label{e.maxmap.C_0}
\mcg_0=\{\maxmap(\x) \setdef \x \in {}^{\nat}2\}.
\end{equation}
Extend $\maxmap$ to $\mathbb{F}\left({}^{\nat}2\right)$ in the unique possible way to obtain a homomorphism.
Recalling \eqref{e.SuSud}, let us introduce the following notation for sets where surgery affects $\xi(f) = \cofmap \big(\chi(f)\big)$:
\begin{alignat*}{2}
\Sud(f) &:=  f\big[D(f)\big] \cup (\cofmap(f)^{-1} \circ f) \big[D(f)\big] &&= \Sud\big(\cofmap(f), f, D(f)\big),\\
\Su(f) &:= D(f) \cup \Sud(f) &&= \Su\big(\cofmap(f), f, D(f)\big).
\end{alignat*}

With this notation at our disposal, it will be easier to formulate and explain the proofs of the following propositions.

\medskip

It is useful to give conditions which the map $f \mapsto \set f$ has to satisfy and which imply that $\mcg_0$ as defined above generates a group which is maximal cofinitary.
We do this in the following proposition.
(In this proposition, as in the remainder of the article, we work with $\cofmap, \cfname$, $\cofmapext$, and $\vec I$ as constructed above and in the previous section. For the proof of the \emph{proposition itself} very little is required of these ingredients. 
It is for the \emph{existence} of the map $\setname$ as claimed in the proposition---without which of course the proposition is useless---that we tailored the properties of $\cofmap, \cfname$, $\xi$, and $\vec I$.) 
\begin{proposition}\label{p.bp}
Suppose we have a map
\begin{gather*}
\cofperm \to \infinitesubsets{\nat},\\
f \mapsto \set f
\end{gather*}
such that for all $f,f' \in \cofperm$,
\begin{enumerate}[label=(\Roman*)]
\item\label{e.D.ad}\label{e.D.first} if $f \neq f'$, $\set f \cap \set{f'}$ is finite,
\item\label{e.D.up}  for any $m \in \set f$, if $m \in I_n$ and $f(m) \in I_{n'}$ then $n \leq n'$.
Moreover, $\set f$ meets each component $I_n$ of $\vec I$ in at most one point. %
\item\label{e.D.spaced} If $\neg\kappa_{\setname}(f)$, $\set f$ is $(\cofmapext(f),f)$-spaced,
\item\label{e.D.superspaced} If $h \in \cofperm$ and $\kappa_D(h)$ then $h \res Y = \cofmap(w)\res Y$ for some $Y \in \infinitesubsets{\set{h}}$ and $w = x_l \hdots x_0 \in \mathbb{F}\left({}^{\nat}2\right)$ such that 
$Y \cap \Su(f_j) = \emptyset$ for each $j\leq l$ with $x_j \in \ran(\chi)$ and $f_j := \chi^{-1}(x_j)$ such that $f_j \neq h$.\footnote{Equivalently, one could replace ``such that $f_j \neq h$'' by ``such that $\neg\kappa_{\setname}(f_j)$'' here.}
\end{enumerate}
Then the group (call it $\mcg$) generated by the set 
$\mcg_0$ defined as in \eqref{e.mcg.bp.C_0} is maximal cofinitary.  In other words: 
\begin{multline}\label{e.mcg.bp}
\mcg := \la \mcg_0 \ra^{S_\infty} = 
\Big\la
\big(
 \left\{c \in \infG_0 \setdef \neg(\exists f \in \cofperm)\: \big[\xi(f) = c \land\neg\kappa_{\setname}(f) \big]\right\} 
\cup \\
\left\{\cofmapext(f)\glue{\set f} f \setdef f \in \cofperm \land \neg\kappa_{\setname}(f)\right\}
\Big\ra^{S_\infty}
\end{multline}
is a maximal cofinitary group. 
\end{proposition}
If the reader is puzzled by \ref{e.D.superspaced} they should look ahead to Proposition~\ref{p.max} and Remark~\ref{r.superspaced} now.
Observe that \eqref{e.mcg.bp} is well-defined and $\mcg$ is a group since by construction of $\cofmapext$ and by Lemma~\ref{l.permutation} each element of $\mcg_0$ is a permutation of $\nat$. 
The reader may find it helpful to refer to Figure~\ref{sf.coding}.

 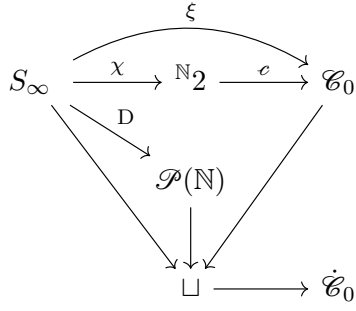
\begin{figure}[h]
  \centering
  \begin{tikzcd}
S_\infty     \ar[r,"\cfname"] \ar[rr,bend left,"\xi"] \ar[dr,"\setname"] \ar[ddr]
 &  {}^{\nat} 2  \ar[r,"\cofmap"]  
& \infG_0  \ar[ddl] \\ 
&\powerset(\nat) \ar[d] &    \\
& \sqcup \ar[r]   & \mcg_0
\end{tikzcd}
  \caption{Coding and catching permutations}
  \label{sf.coding}
  \end{figure}

The above proposition would not be useful if the only way to choose such a map $f \mapsto \set f$ would be to use $\AC$/Zorn's Lemma.
But to the contrary, there is an explicit and purely combinatorial construction of a map $f \mapsto \set f$ with the above properties, without appealing to $\AC$ in any shape or form. 
\begin{lemma}[\ZF]\label{l.D.ZF.intermed}
There is a map $\setname\colon \cofperm \to \powerset(\nat)$ satisfying \ref{e.D.first}--\ref{e.D.spaced} from Proposition~\ref{p.bp} above.
\end{lemma}

\begin{proof}
Let us fix, for the remainder of this article, a bijection 
\begin{gather}
\# \colon {}^{<\omega}2 \to \nat,\label{e.2N}\\
\x^* \mapsto \#(\x^*) \notag
\end{gather}
To achieve \ref{e.D.ad} we let
\begin{equation*}%
\adsetname_0(f) = \{\#\big(\cf{f} \res k\big) \setdef k\in \nat\}
\end{equation*}
whence $f \neq f' \Rightarrow \lvert \adsetname_0(f)\cap\adsetname_0(f')\rvert < \omega$.

To also achieve \ref{e.D.up} we define 
\begin{equation}\label{e.up}
\adsetname_1(f) : = 
\left\{\min\left( (m_n,m_{n+1}] \setminus \bigcup_{k<n} f^{-1}[I_k] \cup \{f^{-1}(m_n)\}\right)\;\middle|\; n\in\adsetname_0(f)\right\}.
\end{equation}
This set is infinite by \ref{r.I_n} in our construction of $\mathcal C$ (see p.~\pageref{r.I_n}).
Note the use of the open interval $(m_n, m_{n+1}]$; this is a mere convenience, and only relevant when we re-use the present definitions in later propositions; see Remark~\ref{r.even} for the reason.

To ensure \ref{e.D.spaced} it is enough to further thin out $\adsetname_1(f)$ to a subset
which we will call $\adsetname_2(f)$. 
In fact, since the requirement in \ref{e.D.spaced} is conditional on $f$ being caught on the final set $\set f$ which we are in the process of constructing, we can do away with an easy case:
If
\begin{equation}\label{e.spacing.works}
\big\{m \in \setname_1(f) \setdef f(m) \neq m \;\land\; f(m) \neq \xi(f)(m)\big\} \text{ is finite,}
\end{equation}
simply let $\setname_2(f) = \setname_1(f)$. Then, as $\set f \in \infinitesubsets{\setname_2(f)}$, $\kappa_{\setname}(f)$ will hold.

If otherwise the set in \eqref{e.spacing.works} is infinite, we thin out as follows: 
Let %
\begin{multline*}
m^f_k = \text{ least $m \in \adsetname_1(f) \setminus \fix(f) $ such that } f(m) \neq g(m) \text{ and }\\
 m > h(m^f_{l}) \text{ for each }l<k, 
 \text{ and } 
 h \in H \cup H^{-1}
\end{multline*}
where 
\begin{align*}
g &:= \cofmapext(f),\\
H &:= \{f, g^{-1} \circ f, f^{-1} \circ g^{-1} \circ f, f^{-1} \circ g \circ f \}.
\end{align*}
Note that $H$ is the set from \ref{i.avoid.1} in the definition of $(g,f)$-spaced (see p.~\pageref{i.avoid.1}).
Now let 
\[
\adsetname_2(f) := \{m^f_k \setdef k\in\nat\}. 
\]
It is clear that the condition in \ref{i.avoid.1} for $m \neq m'$ is enforced by the second line in the above definition of $m^f_k$; for $m = m'$, use the first line of said definition and the fact that $g$ has \emph{no} fixed points.
Thus, $\setname_2(f)$ is $\big(\xi(f),f\big)$-spaced. 
\end{proof}

We shall reuse the notation $\adsetname_2(f)$ in the next section to construct a maximal cofinitary group which is Borel, and also one which is even arithmetical. Therefore, we pause and gauge of the definitional complexity of the map $\setname_2$. 
\begin{lemma}\label{l.D_2.comp}
Given $f \in S_\infty$, the set $\setname_1(f)$ is computable in $f$, and $\setname_2(f)$ is computable relative to an oracle consisting of $f$ and the truth value of \eqref{e.spacing.works}.
Therefore, $\setname_2(f)$ is uniformly $\Delta^3_0(f)$.
\end{lemma}
\begin{proof}
The proof is straightforward. 
\end{proof}

Before we finish the construction of the map $\setname$ satisfying Proposition~\ref{p.bp} we discuss the most involved requirement, Item \ref{e.D.superspaced}.
\begin{remark}\label{r.superspaced}
We sketch how Requirement \ref{e.D.superspaced} ensures that $\mcg$ is maximal (more detail is found in the proof of Proposition~\ref{p.max}):
Suppose we are given $h \in S_\infty$ and want to show that $\mcg_0\cup\{ h \}$ is not contained in a cofinitary group. 
As explained at the beginning of Section~\ref{s.bp}, the $\neg\kappa_D(h)$ case will be easy, so let us suppose $\kappa_D(h)$ holds.
Fix $w = (x_l)^{i_l} \hdots (x_0)^{i_0}$ and an infinite set $Y_0 \in \infinitesubsets{\set h}$ such that $h \res Y_0 = \cofmap(w) \res Y_0$.
We know $\maxmap(w) \in \mcg_0$, but we must still show $\cofmap(w) \res Y = \maxmap(w) \res Y$ for some $Y\in \infinitesubsets Y_0$.
The existence of such $Y$ is exactly what \ref{e.D.superspaced} requires. 

How will we guarantee this?
Such $Y$ exists unless for all but finitely many $m \in Y_0$, the path of $m$ under $\cofmap(x_l)^{i_l}, \dots, \cofmap(x_0)^{i_0}$ meets some $\Su\big(\chi^{-1}(x_j)\big)$; this is the set where $\maxmap(x_j)$ potentially differs from $\cofmap(x_j)$.
In fact, by \ref{e.D.ad} the task is reduced to ensuring\footnote{This case was overlooked in an earlier version of this article. Thanks to Severin Mejak for noticing the gap.}  the sets $\Sud\big(\chi^{-1}(x_j)\big)$ avoid said path, for all $j\in J$.

That is, the potential problem is a set $U=\{f_j \setdef j \in J\} \subseteq S_\infty$
where ``the $\Sud(f_j)$ are too greedy'' in the sense that $\bigcup_{j\in J} \Sud\big(f_j\big)$ almost covers $Y_0$ (that is, with only finitely many exceptions).
Let us call such $U$ an \emph{uncooperative set for $h$ and $w$}.

To  ensure that \ref{e.D.superspaced} holds, we approach the above situation from the point of view of a potential element of an uncooperative set $U$. 
Given $f$ we shall be able to detect that $f$ is one of the permutations from a potentially uncooperative set $U=\{f_j \setdef j\in J\}$ for some $h$ and $w$. 
In fact, we arrange---by making $\setname(f)$ sparse---that there is at most one $h$ and $w$ for which this can occur.
We then make each set $\set{f_j}$ so sparse that $Y_0 \setminus \bigcup_{j\in J} \Sud(f_j)$ remains infinite, for $Y_0$ as above. 
For this, $f=f_j$ has to take into account $h$ and $w$ as well as the other permutations from $U$, that is, the thinning out has to be coordinated (or \emph{cooperative}) among $U$. 
This is achieved using a ``semaphore'' which reserves some points of $Y_0$ for the catching of $h$.
Crucially, all the relevant information (that is, $h$, $w$, and the set $U$ of all participants in the potential conflict) can be reconstructed from each single $f \in U$, so they will indeed use the same semaphore.
\end{remark}

\begin{lemma}[\ZF]\label{l.D.ZF}
There is a map $\setname\colon S_\infty \to \powerset(\nat)$ which in addition to \ref{e.D.first}--\ref{e.D.spaced} also satisfies \ref{e.D.superspaced} from Proposition~\ref{p.bp} above.
\end{lemma}
Before we prove the lemma, we introduce some notation which will be useful throughout this article.
Firstly, we define a strict partial order on $\nat$:
Let
\[
m \porange{h} m' \stackrel{\text{def}}{\iff}  s \subsetneq s' \text{ for the unique $s,s' \in {}^{<\omega}\nat$ s.t.\ } 
m \in I_{\#(s)} \land m' \in I_{\#(s')}.
\]

Secondly, given $m,m' \in \nat$
(and recalling the map $c_n$ from Proposition~\ref{p.groundwork}) define 
\[
\diffword(m,m') = \begin{cases}  \text{the unique element $w \in W_n$
such that $\fcode_n(w)(m)=m'$, if such exists,}\\
\text{$\uparrow$ (i.e., remains undefined) otherwise.} 
\end{cases}
\] 

For aesthetic reasons, we make the next two of the current series of definitions slightly more general than is presently needed (i.e., for $h\in \parinj(\nat,\nat)$ and not just $h\in S_\infty$).

Thirdly, given $h\in \parinj(\nat,\nat)$, 
we define a strict partial order on $\nat$.
Let
\[
m_0 \po{h} m_1
\]
if and only if $m_0 <m_1$, 
and for each $i \in \{0,1\}$,
$w_i := \diffword\big(m_i,h(m_i)\big) \in \mathbb{F}\big({}^{n_i}2\big)$ is defined and
\[
w_0 =r^{n_1}_{n_0}(w_1).
\]
Finally, given a partial order $\prec$ we shall say a set $X$ is \emph{$\prec$-homogeneous} \emph{iff} either $X$ consists only of $\prec$-incomparable elements, 
or else $X$ is totally ordered by $\prec$.

\begin{proof}[Proof of Lemma~\ref{l.D.ZF}]
We start with the map $\setname_2$ constructed in Lemma~\ref{l.D.ZF.intermed} which already satisfies \ref{e.D.first}--\ref{e.D.spaced} and ``thin out'' several more times to ensure \ref{e.D.superspaced}. 

Firstly, if $\kappa\big(\setname_2(f),f\big)$, we simply let $\setname(f) = \setname_2(f)$.
Next, find a map 
\[
\setname_3\colon \{f\in \cofperm \setdef \neg\kappa\big(\setname_2(f),f\big)\}\to \powerset(\nat)
\]
such that $\setname_3(f) \in \infinitesubsets{\setname_2(f)}$ and  $f\left[\setname_3(f)\right]$ 
is $\porange{f}$-homogeneous for each $f \in \dom(\setname_3)$.
To this end, consider the following relation on $\cofperm \times \powerset(\nat)$:
\[
R(f,D') \stackrel{\text{def}}{\iff} \big( D' \in \infinitesubsets{\setname_2(f)} \land f[D'] \text{ is $\porange{f}$-homogeneous}\big).
\]
By Ramsey's Theorem, for each $f \in \cofperm$ there is $D'$ such that $R(f,D')$.
As $R$ is $\Pi^1_1$ (even arithmetical, as is straightforward to verify) a map $\setname_3$ as desired exists (provably in $\ZF$) by $\Pi^1_1$-Uniformization.\footnote{We will soon show that in fact, the map $f \mapsto \setname_3(f)$ can be chosen to be arithmetical.}
Given  $f \in \dom(\setname_3)$,
by construction, for at most one 
$h \in {}^{\nat}\nat$ 
does 
\begin{equation}\label{e.h_f}
\left(\exists X \in \infinitesubsets{\setname_3(f)}\right)\;  f[X] \subseteq \Isat\big(\setname_2(h)\big), 
\end{equation}
hold. 
Let us therefore write $h_f$ for it, and say ``$h_f$ exists'' to mean ``there exists $h \in {}^{\nat}\nat$ satisfying \eqref{e.h}''.
Clearly $h_f$ is then definable from $f$.

By the same argument as above, we can find a map $\setname_4\colon \dom(\setname_3) \to \powerset(\nat)$ such that $\setname_4(f) \in \infinitesubsets{\setname_3(f)}$ and  if $h_f$ is defined, $f\big[\setname_4(f)\big]$ is $\pow{h_f}$-homogeneous. 
For any $f \in \cofperm$ such that $\kappa_{\setname_2}(f)$ and $h_f$ is defined, by construction, there is at most one $w \in \mathbb{F}\left({}^{\nat} 2\right)$ such that
\begin{equation}\label{e.w}
\left(\exists X \in \infinitesubsets{\setname_4(f)}\right)\;  h_f \res f[X] = \cofmap(w)   \res f[X].
\end{equation}
Analogously to the above, let us denote such $w$ by $w_f$ if it exists, and let us express this state of affairs by ``$w_f$ exists''.
(Now $f$ can be an element of a uncooperative set for at most one pair $h$ and $w$---namely $h_f$ and $w_f$.)

Given $f \in \dom(\setname_3)$, if $h_f$ or $w_f$ do not exist, then we can let $\setname_5(f) = \setname_4(f)$.
Now suppose both $h = h_f$ and $w = w_f$ exist %
and write
\begin{equation}\label{e.h}
w = (x_l)^{i_l} \hdots (x_0)^{i_0}
\end{equation}
where each $x_j \in {}^{\nat} 2$ and $i_j \in \{-1,1\}$. 
Let $J$ be the set of $j \leq l$ such that
$\x_j \in \ran(\chi)$ and $\chi^{-1}(x_j) \neq h$,
and for each $j\in J$, let\footnote{It would be enough to consider $j$ such that $x_j \in \chi\left[\{f \in S_\infty \setdef \kappa_{\setname_2}(f)\}\right]$.}
\[
f_j := \chi^{-1}(x_j).
\] 
As described in Remark~\ref{r.superspaced}, catching of $h$ may fail because $\{f_j \setdef j \in J\}$ form an uncooperative set. 
We now describe a ``semaphore'' which reserves some points of each $\setname(f_j)$,
 thought of as a scarce resource, for the catching of $h$.
 (Note that if it should be the case that $f \notin \{f_j \setdef j \in J\}$, then there is no uncooperative set in which $f$ participates, and we can let $\setname_5(f) = \setname_4(f)$ and are done. But it doesn't hurt to follow the procedure below for every $f$.)

For the final step, we shall use the shorthand
\[
\Sud_4(f) := \Isat\big( f \big[\setname_4(f)\big]\big) 
\] 
Recursively define a sequence $\bar y = (y_n)_{n \in \nat}$. \label{p.def.bar.y} 
This sequence only depends on $f$ only through $h = h_f$ and $w = w_f$,
therefore we shall also write $\bar y^{h,w} = (y^{h,w}_n)_{n \in \nat}$ for it.
To start the induction, 
let
\[
y_0 =  
\text{ the least } 
y \in \setname_2(h)
\text{ such that } h(y) = \cofmap(w) (y)
\]
Now suppose $n \in \nat\setminus 1$ and $y_{n-1}$ is already defined.
Let 
\begin{multline*}
y_n =  
\text{ the least } 
y \in \setname_2(h)
\text{ such that } h(y) = \cofmap(w) (y) \text{ and } (\forall j \in J)\;\\
 y \in \Sud_4(f_j) \Rightarrow \Big[ (\exists m \in \nat)\; y_{n-1} < m < y \land m \in \Sud_4(f_j)  \Big]
\end{multline*}
That is, $y$ is protected from being used by $f_j$ provided $f_j$ has been able to use a point $m$ previously, earlier than its present request at $y$ but, in case $n>0$, after the previous point $y_{n-1}$ reserved for $h$ (where potentially, we also had to deny $f_j$ access).
Define 
\[
\setname_5(f) = \Big\{m\in\setname_4(f) \setdef f(m) \notin I\left(\left\{y^{h_f, w_f}_n \setdef n\in\nat\right\}\right)\Big\}.
\]
By construction, $\setname_5(f)$ is infinite. 
(Note that no similarly easy construction would be possible if we hadn't arranged that there is at most one pair $h_f, w_f$ for which $f$ is potentially uncooperative). 
Finally, we conclude the case of $f \in \cofperm$ such that $\neg\kappa\big(\setname_2(f), f\big)$ by defining
\[
\setname(f) =\setname_5(f). 
\]
Then \ref{e.D.superspaced} holds: Given an arbitrary $h \in S_\infty$ and $w$ such that $h$ and $\cofmap(w)$ agree on an infinite subset of $\setname(h)$, write $w$ as \eqref{e.w} above and let $\{f_j \setdef j\in J\}$ be defined as above. 
We show there is
an infinite set $Y \subseteq \setname_2(h) = \setname(h)$ disjoint from each $\Su(f_j)$; namely, let
$Y := \ran\left(\bar y^{h,w}\right) \setminus \bigcup_{j\in J} \setname_2(f_j)$. 
By construction,  for each $j \in J$
\[
\ran\left(\bar y^{h,w}\right) \cap \Sud(f_j) = \emptyset
\]
and so since $\Su(f_j) \subseteq \set{f_j} \cup \Sud(f_j)$, $Y$ is disjoint from $\Su(f_j)$.
Note that $Y$ is infinite %
by \ref{e.D.ad}.
\end{proof}

\medskip

We now prove Proposition~\ref{p.bp}. 
The proof will take up the remainder of this section and the next section and is split into two further propositions, the first of which has the purpose of verifying maximality.

\begin{proposition}[\ZF]\label{p.max}
For any $h\in S_\infty$ there is $c \in \mcg$ such that 
$\{n\in\nat\setdef h(n)=c(n)\}$ is infinite.
In particular, provided we can show that the group $\mcg$ is cofinitary, $\mcg$ will be \emph{maximal} cofinitary.
\end{proposition}
\begin{proof}
Let $h\in S_\infty$ be given. 
Suppose first that $h$ is not caught, that is $\neg\kappa_{\setname}(h)$ holds, or in more detail, $\kappa\big(\set h,h\big)$ from Eq.~\eqref{e.kappa} fails.
Then letting $x:=\chi(h)$, by definition of $\maxmap$, 
$h \res \set h = \maxmap(x)\res \set h$, whence $h$ agrees on an infinite set with the element $c := \maxmap(x)$ of $\mcg$.

Now consider the case that $h$ \emph{is} caught---that is, $\kappa_{\setname}(h)$ or equivalently, $\kappa(\set h,h)$ from Eq.~\eqref{e.kappa} holds.  
Let us fix a word $w \in \mathbb{F}\left({}^{\nat} 2\right)$ and an infinite set $Y\subseteq \nat$ witnessing
\ref{e.D.superspaced}. 
Then,
\begin{equation}\label{e.f}
h \res Y = \cofmap(w) \res Y.
\end{equation}
Let us write 
\[
w = (\x_l)^{i_l} \hdots (\x_0)^{i_0},
\]
let $J$ be the set of $j \leq l$ such that
$\x_j \in \ran(\chi)\setminus\{\chi(h)\}$,
and let
\[
f_j := \chi^{-1}(x_j)
\]
for each $j\in J$. 
By choice of $Y$---that is, by \ref{e.D.superspaced}---we have
\begin{equation}\label{e.act}
\maxmap(\x_j) \res Y = \cofmap(\x_j) \res Y
\end{equation}
for any $j \in J$, 
since surgery is only applied to points in $\Su(f_j)$, and this set is disjoint from $Y$.
Note that if $x_j = \chi(h)$, Equation~\eqref{e.act} is also true by definition of $\maxmap$ and surgery;
and likewise, if $x_j \notin \ran(\chi)$ is Equation~\eqref{e.act} is true by definition of $\maxmap$.
Thus, Equation~\eqref{e.act} holds for all $j \leq l$, 
whence also $\cofmap(w)\res Y= \maxmap(w) \res Y$. 
From this and \eqref{e.f} we infer that $h$ agrees on $Y$ with $\maxmap(w)$.
\end{proof}

\subsection{Cofinitariness}\label{s.cof}

In this section we prove that the group $\mcg$ constructed in the previous section---or more precisely, any group constructed as in Proposition~\ref{p.bp}---is cofinitary.

\begin{proposition}\label{p.cof.gen}
Under the same assumptions as in Proposition~\ref{p.bp},
 $\mcg$ as defined there,
is a cofinitary group.
\end{proposition}

\begin{proof}%
Suppose $c \in \mcg$
and $c$ has infinitely many fixed points.
Let $l \in \nat$ be minimal such that $c$ arises via composition from a sequence of length $l$ of generators/inverses of generators.
Supposing towards a contradiction $l>0$,
choose
$c_0, \hdots, c_{l-1} \in \mcg_0$ and $i_0, \hdots, i_{l-1} \in \{-1,1\}$ such that 
\begin{equation}\label{e.g}
c = (c_{l-1})^{i_{l-1}} \circ \hdots \circ (c_0)^{i_0}.
\end{equation}
By minimal choice of $l$,
$
(c_{l-1})^{i_{l-1}}\hdots (c_0)^{i_0}
$
is reduced in the usual sense that it contains no subwords of the form $c^{-i}c^{i}$ with $c \in \mcg_0$, i.e., it is reduced as a word in $\mathbb{F}(\mcg_0)$.
For each $i < l$ we can pick $\x_i \in {}^{\nat}2$ so that 
\[
c=\maxmap\left((\x_{l-1})^{i_{l-1}} \hdots (\x_0)^{i_0}\right),
\]
or in other words, so that
either 
\[
c_i = \cofmap(\x_i)  
\]
if $\x_i \notin \ran(\chi)$ or $\kappa_{\setname}(\invmap(\x_i))$, or otherwise if $\x_i \in \ran(\chi)$ and $\kappa_{\setname}(\invmap(\x_i))$, then 
\begin{equation}\label{e.c.subst}
c_i = \cofmap(\x_i) \glue{\set{\invmap(\x_i)}} \invmap(\x_i).
\end{equation}
In the second case, let us write 
\begin{gather*}
f_i :=\invmap(\x_i).
\end{gather*}
Since the word on the right in \eqref{e.g} is
reduced with respect to the rules in $\mathbb{F}(\mcg_0)$, 
\begin{equation*}
w := (\x_{l-1})^{i_{l-1}} \hdots (\x_0)^{i_0}
\end{equation*}
is in reduced form as a word in $\mathbb{F}\left({}^{\nat} 2\right)$. 

Let $F$ be a tail segment of $\fix(c)$ such that 
for all $m \in F$ and for all points $m'$ in the path under $w$ of $m$, $m'$ lies in at most one of the sets
$\set{f_i}$, for any $i<l$ such that $f_i$ is defined. 
This is possible by \ref{e.D.ad}.

For any $m\in F$ and $j  < l$ such that $c_j$ has the form as in \eqref{e.c.subst}, the permutation 
\[
\cofmap(\x_j) \glue{\set{f_j}}f_j (m)
\] 
acts in the path under $w$ of each element of $F$ as one of 
\begin{gather*}
\cofmap(\x_j)(m), \\
\cofmap(\x_j)^2(m), \\
f_j(m)\text{, or}\\ 
\big(\cofmap(\x_j)\circ {f_j}^{-1}\big)(m)
\end{gather*}
as in \eqref{e.def.glue}. 
Thus, for each $m \in \fix(c)$ we can find $l(m)\leq 2l$, $\dot{c}^m_j$, and $i^m_j$ for $j < l(m)$ such that
\begin{equation*}
c(m) = \left( \dot{c}^m_{l(m)-1} \circ \hdots \circ \dot{c}^m_0 \right)(m)
\end{equation*}
where for each $j < l(m)$
\[
\dot{c}^m_j = \begin{cases}
\cofmap(\x^m_j)^{i^m_j} &\text{ or }\\
(f^m_j)^{i^m_j}
\end{cases}
\]
with $\x^m_j \in \{\x_{l-1}, \hdots, \x_0\}$ and $f^m_j :=\invmap(\x^m_j)$ when $\x^m_j \in \ran(\chi)$ and the above equation calls for $f^m_j$ to be defined; that is, in this case $f^m_j = f_i$ for some $i < l$.

Write
\[
w^m := \left(\x^m_{l(m)-1}\right)^{i^m_{l(m)-1}}  \hdots  \left(\x^m_{0}\right)^{i^m_{0}}.
\]
Note again that the length $l(m)$ of this new word $w^m$
is bounded by the definition of surgery, namely we have $l(m) \leq 2 l$.
Since there are only finitely many possible such substitutions (each $\x^m_j$ being chosen from $\{\x_i \setdef i<l\}$) we can 
write $F$ as a finite union of sets on each of which $w^m$ is constant in $m$.
Let  $F^* \subseteq F$ be one such set which is infinite.
Replacing each superscript ``$m$'' by ``$*$'',  we write
\begin{gather*}
l(m)= l^*,  \\
c^m_j =  c^*_j, \\
\x^m_j = \x^*_j,\\
i^m_j = i^*_j, \\
f^m_j = f^*_j,
\end{gather*}
for all $m \in F^*$ and all $j < l^*$. By construction,  
\begin{gather*}
c^*_j = \begin{cases}
\cofmap(\x^*_j)^{i^*_j} &\text{ or }\\
\invmap(\x^*_j)^{i^*_j} = f^*_j, 
\end{cases}
\end{gather*}
for all $m \in F^*$ and all $j < l^*$. 
Moreover, 
\[
\big( c^*_{l^*-1} \circ \hdots \circ c^*_0 \big)\res F^* = \maxmap(w) \res F^*= c \res F^* = \id_{F^*}.
\]
Finally, we also write %
\[
w^*:=(\x^*_{l^*-1})^{i^*_{l^*-1}}\hdots (\x^*_0)^{i^*_0}.
\] 

\begin{claim}\label{c.w*}
The word $w^*$ reduces to $\emptyset$ in $\mathbb{F}\left({}^{\nat} 2\right)$.
\end{claim}
\begin{proof}[Proof of claim]
Suppose otherwise that as an element of $\mathbb{F}\left({}^{\nat} 2\right)$, the word $w^*$ reduces to ${v}$ and ${v} \neq \emptyset$.
We will derive a contradiction.

Fix $\bar l \in \nat$ and a sequence $j(0), \hdots, j(\bar l-1)$ so that we may write the word ${v}$ as
\[
{v}= \left(\x^{*}_{j(\bar l-1)}\right)^{i^*_{j(\bar l-1)}} \hdots \left(\x^{*}_{j(0)}\right)^{i^*_{j(0)}}.
\]
For now, fix $m \in F^*$ arbitrarily. 
Let us write the path of $m$ under the word ${v}$ as $m(0), m(1), \hdots, m(\bar l)$, where $m(0)=m$ and $m(\bar l)=\cofmap^*({v}) (m) = m$, and  
\begin{equation*}
m(k+1) = 
c^*_{j(k)} \big(m(k)\big)
\end{equation*}
for each $k < \bar l$.
Let us look at the subword corresponding to a part of the path which is spent in the interval from our partition $\vec I$ with lowest possible index: 
I.e.,
let 
\[
K^m=[k^m_0,k^m_1]
\]
be a non-empty interval in $\Int/\bar l\Int$ such that\footnote{We conveniently identify indices along the path with integers modulo $\bar l$; alternatively, one can ``shift'' the path by taking a cyclic permutation of the words $w$ and $w^*$ to guarantee $0 \leq k^0_m \leq k^m_1 \leq \bar l$.}
for all $k \in K^m$, $m(k) \in I_{n(m)}$ where
\[
n(m) := \min\{n \setdef (\exists k\leq \bar l)\; m(k) \in I_n\};
\]
further let us suppose that $K^m$ is maximal in the sense that (working modulo $\bar l$) either $K^m = [0,\bar l]$ 
or $m(k^m_0-1) \notin I_{n(m)}$
and $m(k^m_1+1) \notin I_{n(m)}$.
\begin{subclaim}\label{sc.fp}
It holds that $m(k^m_1)=m(k^m_0)$.
\end{subclaim}
\begin{proof}[Proof of Subclaim]
The first possibility is that the entire path of $m$ under $v$ lies within $I_{n(m)}$.
In this case, we may assume $k^m_0 = j(0)$ and $k^m_1 = j(\bar l-1)$ and $m(k^m_1)=m(k^m_0)=m$.

If on the other hand the path enters $I_{m(n)}$ from another interval component of $\vec I$, 
since by choice of $m(n)$ this second interval comes later in $\vec I$, the path must enter via an application of some $(f_i)^{-1}$, 
where $i$ is unique such that $\set{f_i}\cap I_{m(n)} \neq \emptyset$, and $m(k^m_0)$ is the unique point in this intersection. 
By the same argument,  $m(k^m_1)$ must also be equal to this unique point in $\set{f_i}\cap I_{m(n)}$.
\renewcommand{\qedsymbol}{{\tiny Subclaim \ref{sc.fp}.} $\Box$}
\end{proof}\
Let 
\[
{\tilde K}^m = [{\tilde k}^m_0,{\tilde k}^m_1]
\]
 be a sub-interval of $K^m$ which is non-empty and minimal 
with the property that $m({\tilde k}^m_1)=m({\tilde k}^m_0)$. 
Shrinking $F^*$ to an infinite subset $\tilde F$ if necessary, we may assume that ${\tilde K}^m$ is independent of $m$;
let us suppose for all $m\in \tilde F$,
\[
\tilde K^m = \tilde K =[{\tilde k}_0, {\tilde k}_1].
\] 
Now consider the word 
\[
{\tilde v} := {v} \res {\tilde K} = \left(\x^{*}_{j({\tilde k}_1)}\right)^{i^*_{j({\tilde k}_1)}} \hdots \left(\x^{*}_{j({\tilde k}_0)}\right)^{i^*_{j({\tilde k}_0)}}, 
\]
corresponding to the permutation
\[
c^*_{j({\tilde k}_1)} \circ \hdots \circ c^*_{j({\tilde k}_0)} .
\]
Let us emphasize again that by construction,
\[
\tilde F \subseteq \fix\left(c^*_{j({\tilde k}_1)} \circ \hdots \circ c^*_{j({\tilde k}_0)}\right),
\]
that $\tilde F$ is infinite, and that by minimality of $\tilde K$, %
for any $k, k' \in [{\tilde k}_0,{\tilde k}_1]$ such that $k < k'$ and $\{k,k'\}\neq \{{\tilde k}_0,{\tilde k}_1\}$, and for any
$m \in    c^*_{j(k-1)} \circ \hdots \circ c^*_{j(\tilde k_0)} [\tilde F]$ (for $k>0$) resp.\ any $m\in \tilde F$ (when $k=\tilde k_0$),
\[
m \neq c^*_{j(k')} \circ \hdots \circ c^*_{j(k)}  (m). 
\]

We now begin with a series of subclaims which culminate in the proof of the assertion that ${\tilde v}=\emptyset$, contradicting the choice of ${\tilde v}$.
\begin{subclaim}\label{c.f.f.inv}
For at most one $j = j(k)$ with $k \in [{\tilde k}_0, {\tilde k}_1)$ is it the case that $c^*_{j}=f^*_{j}$ or  $c^*_{j}=(f^*_{j})^{-1}$.
\end{subclaim}
\begin{proof}[Proof of Subclaim]
Suppose otherwise, fix distinct $k$ and $k'$ from $\tilde K$ such that $j = j(k)$ and $j' = j(k')$ constitute a counterexample to the claim, i.e., $c^*_j \in\{ f^*_j,(f^*_j)^{-1}\}$ and $c^*_{j'} \in\{ f^*_{j'},(f^*_{j'})^{-1}\}$.
Since we have chosen $F$ so that the path of each of its elements passes though at most one of transmutation site, we have $f^*_j = f^*_{j'}$. 
Thus one of the following configurations occurs in such a path under ${\tilde v}$:
\[
\begin{tikzcd}
\dots m(k'+1)  & \arrow{l}[swap]{(f^*_j)^{-1}} m(k')    &  \arrow{l}[swap]{\cofmap(\vec \x)} m(k+1) & \arrow[l, swap, "f^*_j"] m(k)\dots
\end{tikzcd}
\]
or
\[
\begin{tikzcd}
\dots m(k'+1)  & \arrow{l}[swap]{f^*_j} m(k')    &  \arrow{l}[swap]{\cofmap(\vec \x)} m(k+1) & \arrow[l, swap, "f^*_j"] m(k)\dots
\end{tikzcd}
\]
where in the first case, $\vec \x \neq \emptyset$ because $w^*$ is reduced. 
The first is impossible since then $m(k+1)=m(k')$, contradicting our assumption that for no proper subword of ${\tilde v}$ does the corresponding path segment have a fixed point.
The second is also impossible, since then $m(k'+1)=m(k+1)$, leading to the same contradiction.
\renewcommand{\qedsymbol}{{\tiny Subclaim \ref{c.f.f.inv}.} $\Box$}
\end{proof}

\begin{subclaim}\label{c.f.f.inv.1}
It is impossible that $c^*_j$ be $(f^*_j)^{i^*_j}$ for exactly one $j$ as in the previous claim.
\end{subclaim}
\begin{proof}[Proof of subclaim]
Otherwise,
letting $j$ be a counterexample, the path of any element of $\tilde F$ is of the following form:
\begin{equation*}%
\begin{tikzcd}
m(j({\tilde k}_0))  = m(j({\tilde k}_1)) & \arrow{l}[swap]{\cofmap(\vec \x_1)}  m(k+1)  & \arrow{l}[swap]{(f^*_j)^{i^*_j}} m(k)    &  \arrow{l}[swap]{\cofmap(\vec \x_0)} m(j({\tilde k}_0)) 
\end{tikzcd}
\end{equation*}
with $m(k) \in \set{f^*_j}$, for appropriately chosen $\vec \x_0, \vec \x_1 \in \mathbb{F}\left({}^{\nat} 2\right)$ 
and therefore 
\begin{equation}\label{e.one.f.1}
(f^*_j)^{i^*_j}(m) = \cofmap(\vec \x_0 \vec \x_1)(m)
\end{equation}
for infinitely many $m \in \set{f^*_j}$. 
Thus, $\kappa_{\setname}(f^*_j)$. 
But this contradicts that
by assumption, $f^*_j = f^i$ for some $i < l$ with $\neg\kappa_{\setname}(f_i)$.
\renewcommand{\qedsymbol}{{\tiny Subclaim \ref{c.f.f.inv.1}.} $\Box$}
\end{proof}

\begin{subclaim}\label{c.empty}
It must be the case that ${\tilde v}=\emptyset$.
\end{subclaim}
\begin{proof}[Proof of subclaim]
By the previous two claims, all $c^*_j$ are of the form $\cofmap(\x^*_j)^{i^*_j}$.
Therefore,
\[
\cofmap({\tilde v}) (m) = m
\]
for all $m\in \tilde F$. %
But this is only possible if ${\tilde v}$ reduces to $\emptyset$ in $\mathbb{F}\left({}^{\nat} 2\right)$
because $\infG$ is cofinitary and $\cofmap$ is injective. 
\renewcommand{\qedsymbol}{{\tiny Subclaim \ref{c.empty}.} $\Box$}
\end{proof}
With this we reach a contradiction, since by assumption, ${\tilde v}$ is a non-trivial subword of the word ${v}$ obtained by reducing $w^*$.
\renewcommand{\qedsymbol}{{\tiny Claim \ref{c.w*}.} $\Box$}
\end{proof}

We have shown that $w^*$ reduces to $\emptyset$. With the next claim, we reach the desired contradiction and finish the proof of the proposition.
\begin{claim}\label{c.reduced}
It must be the case that already $w=\emptyset$. 
\end{claim}

\begin{proof}[Proof of claim]
Suppose otherwise.
We first consider the case that $w^*$ contains a subword of the form 
\[
(f^*_j)^{-1}f^*_j
\] 
for some $j<l$.
By the definition of $\maxmap$ this subword can only arise via substitution (in the path of elements of $Y$) of a subword of $w$ of the form
\begin{equation*}
(\x^*_j)^{-1}\x^*_j  
\end{equation*}
(substituting each $\maxmap(\x^*_j)$ by $f^*_j$) which is impossible as we have assumed no such subwords occur in $w$;
or via substitution from a subword of $w$ of the form
\begin{equation}\label{e.sw}
\x^*_j \x^*_j,
\end{equation}
substituting $\maxmap(\x^*_j)$ on the right-hand by $f^*_j$, and 
substituting $\maxmap(\x^*_j)$ on the left by $\cofmap(\x^*_j) (f^*_j)^{-1}$.
Therefore, the subword \eqref{e.sw} of $w$ via substitution gives rise to the following subword  of $w^*$:
\begin{equation}\label{e.subst.sw}
\cofmap(\x^*_j) (f^*_j)^{-1} f^*_j
\end{equation}
But since $w^*$ reduces to $\emptyset$ by Claim~\ref{c.w*},  
the occurrence of  
$\cofmap(\x^*_j)$ on the left-hand in \eqref{e.subst.sw} must cancel, so the word in \eqref{e.subst.sw} can be extended to a subword of $w$ of the form 
\[
\cofmap(\x^*_j)^{-1}\cofmap(\x^*_j) (f^*_j)^{-1} f^*_j 
\]
with the left-most letter coming from a substitution of $(\x^*_j)^{-1}$ by $\cofmap(\x^*_j)^{-1}$ or $\cofmap(\x^*_j)^{-2}$.
Therefore, the letter immediately to the left of the subword \eqref{e.sw} in $w$ must be $(\x^*_j)^{-1}$.
This is a contradiction since we have assumed $w$ to be reduced, so no adjacent $\x^*_j$ and $(\x^*_j)^{-1}$ occur in $w$. 

Next, let us consider the case that $w^*$ has a subword of the form $\cofmap(\x^*_j)^{-1}\cofmap(\x^*_j)$. Such a word can only arise from substituting $(\x^*_j)^{-1}\x^*_j$ via the definition of $\maxmap$, so again, this stands in contradiction to the assumption that $w$ be reduced.

Analogous arguments go through by symmetry if $w^*$ has a subword of the form $f^*_j(f^*_j)^{-1}$ or 
$\cofmap(\x^*_j)\cofmap(\x^*_j)^{-1}$.
\renewcommand{\qedsymbol}{{\tiny Claim \ref{c.reduced}.} $\Box$}
\end{proof}
We have shown it must have been the case that $w = \emptyset$ and $l=0$, that is, $c=\id_{\nat}$ to begin with; since $c$ was an arbitrary element of $\mcg$ such that $\fix(c)$ is infinite,
$\mcg$ is cofinitary.
\renewcommand{\qedsymbol}{{\tiny Proposition \ref{p.cof.gen}.} $\Box$}
\end{proof}

\begin{corollary}[\ZF]\label{c.ZF}
There is a maximal cofinitary group.
\end{corollary}

\begin{remark}
It is of course possible to give an upper bound for the definitional complexity of the group obtained in this section; namely, a Boolean combination of $\Sigma^1_2$ statements. 
Since we will construct a MCG of much lower definitional complexity in the next section,
we shall not dwell on this point.
\end{remark}

\section{More complicated construction, simpler definition}\label{s.arithmetical}

The following theorem was shown first by Horowitz and Shelah in \cite{mcg-borel}.
In this section, we finish our proof of their result and also improve their result.

\begin{theorem}\label{t.borel}
There is a Borel (in fact, $\Delta^1_1$) maximal cofinitary group.
\end{theorem}

One of the ways in which the proof given in the previous section differs from Horowitz and Shelah's is that it can be almost effortlessly improved to show Theorem~\ref{t.arithmetic}, that is, the following.

\begin{theorem}\label{t.arithmetical}
There is a finite level Borel (in fact, $\Sigma^0_{<\omega}$, i.e., arithmetical) MCG.
\end{theorem}

These results are provable in $\ZF$; this is obviously true from the proof we give below (but even if we were to give a proof appealing to $\AC$, this appeal could be removed \emph{post facto} by the well-known trick of running the proof in $\eL$ and using absoluteness).

\medskip

For the purpose of a quick proof of Theorem~\ref{t.borel}, let us make the additional assumption that $\xi$ was chosen to be a bijection between $S_\infty$ and $\infG$ (this is not necessary for the proof, but convenient). 
We show that there exists a map $\setname \colon S_\infty \to \infinitesubsets{\nat}$  whose graph is $\Delta^1_1$ and even arithmetical, satisfying \ref{e.D.first}--\ref{e.D.superspaced} as in said Proposition, and so that in addition, $\kappa(\setname(f),f)$ as defined in \eqref{e.kappa}
becomes a Borel---in fact, an arithmetical---property of $f$.

The group $\mcg$ defined from this re-defined map $\setname \colon S_\infty \to \powerset(\nat)$ just as in Proposition~\ref{p.bp} is then maximal cofinitary, by said proposition;  
moreover, it is now easy to see that $\mcg$ is Borel.

In fact we show the following:

\begin{proposition}\label{p.borel}
Suppose  we have maps $\xi$ and $\setname$ satisfying all the assumptions of Proposition~\ref{p.bp} and so that in addition, firstly, both $\setname \colon S_\infty \to \powerset(\nat)$ and $\xi  \colon S_\infty \to \infG$  are analytic maps, secondly,
$\xi$ is a bijection,
and thirdly,
we can find a $\Delta^1_1$ relation $\lambda(X,f)$ on $\powerset(\nat)\times S_\infty$ such that for all $f\in S_\infty$,
\begin{equation}\label{e.bar.lambda}
\lambda(\set f,f) \iff \kappa(\set f,f).
\end{equation}
Then the group $\mcg$ defined as in  Proposition~\ref{p.bp} by \eqref{e.mcg.bp.C_0} and \eqref{e.mcg.bp}, is $\Delta^1_1$ and a MCG.
\end{proposition}
\begin{proof}
That $\mcg$ as in the present proposition is a MCG holds because it also satisfies the assumptions of Proposition~\ref{p.bp};
we show that $\mcg$ is $\Delta^1_1$.

By \eqref{e.bar.lambda},  by definition of $\mcg$, and because $\cofmapext$ is surjective, %
it is obvious that for all $h\in S_\infty$,
\begin{multline}\label{e.Borel}
h\in\mcg \iff (\exists l\in\nat)(\exists g_0,\hdots, g_l \in S_\infty)\big(\exists i_0,\hdots, i_l \in \{1,-1\}\big) \;
h = (g_l)^{i_l} \hdots (g_0)^{i_0} \land \\
(\exists f_0,\hdots, f_l \in S_\infty)
\big(\exists D_0,\hdots, D_l \in \powerset(\nat)\big)\; 
(\forall i \leq l)\; \\
D_i = \set{f_i} \land
\Big[\big(
\lambda(D_i,f_i) \land g_i = \cofmapext(f_i) 
\big) 
\lor \\
\big(
\neg\lambda(D_i,f_i) \land g_i = \cofmapext(f_i) \glue{D_i} f_i 
\big)\Big].
\end{multline}
Since $\lambda(D_i,f_i)$ is $\Delta^1_1$, and since the relations
\begin{gather*}
h = (g_l)^{i_l} \hdots (g_0)^{i_0},\\
g_i = \cofmapext(f_i),\\
g_i = \cofmapext(f_i) \glue{D_i} f_i 
\end{gather*}
are arithmetic---in fact, $\Pi^0_1$---in $h$ and since the map $\setname$ is a $\Sigma^1_1$, clearly, the formula to the right of ``$\iff$'' in \eqref{e.Borel} is $\Sigma^1_1$.

By maximality of $\mcg$ it holds that for any $h\in S_\infty$, 
\begin{multline*}\label{}
h \notin \mcg \iff (\exists g_0, \hdots, g_l \in {}^{\nat}\nat)(\forall j\leq l)\; 
 g_j \in \mcg\land \\
 (\exists i_1, \hdots, i_l)\;  \fix\big(g_l h^{i_l} \hdots h^{i_1}g_0\big) \text{ is infinite,}
\end{multline*}
and so clearly $\mcg$ is also $\Pi^1_1$. Thus, $\mcg$ is $\Delta^1_1$.
\end{proof}

We next show that a map $\setname \colon S_\infty \to \powerset(\nat)$ as in the previous proposition exists.
The construction given in the proof of Lemma~\ref{l.D.ZF} is not sufficient here for two reasons:
Firstly, there is no indication of how we might find the predicate $\lambda$.
Secondly, we did not pay close attention to definability, in particular in how certain homogeneous sets were chosen. 

We now give a similar construction, verifying that the same choice can be made in a 
$\Sigma^1_1(f)$ fashion---in fact, arithmetically-in-$f$.
In fact, this same (second) version of $\setname\colon \cofperm \to \powerset(\nat)$ is used in both Propositions~\ref{p.borel} and~\ref{p.arithm.MCG}, that is,
we re-use it in the construction of an arithmetical MCG.

\medskip

We shall use the following two lemmas:

\begin{lemma}
Suppose we are given $D \subseteq \nat$ and a partial order $\prec$ on $D$.
There is an infinite, uniformly arithmetical-in-$(D,\prec)$ set $H=H(D,\prec)\subseteq D$ which is $\prec$-homogeneous (that is, totally ordered by $\prec$
or consisting of pairwise $\prec$-incomparable elements).
\end{lemma}
\begin{proof}
Define the predicate $\mathcal T = \mathcal T(D,\prec)$ by
\begin{equation}\label{e.T}
\mathcal T :\iff \big(\forall n \in D\big)\big(\exists n' \in D\setminus(n+1)\big)\big(\forall n'' \in D\setminus (n'+1)\big) \;
n' \prec n''.
\end{equation}
Clearly this predicate is arithmetical in $(\prec,D)$.
(The letter $\mathcal T$, i.e., ``T'' in script type stands for ``tangled''.)

\medskip

We can now define $H = H(\prec,D)$ by distinguishing two cases: \label{page.D.2}

\begin{description}
\item[Case 1] $\mathcal T$ holds.
In this case, we can fix $n_0$ such that 
\[
\big(\forall n' \in D\setminus (n_0+1)\big)\big(\exists n'' \in D\setminus (n'+1)\big)\; n' \prec n''.
\] 
It is therefore easy to pick an infinite subset of $D$ consisting of pairwise $\prec$-comparable elements.
Define $m_0, m_1, \hdots$ by induction as follows:
\begin{gather*}
m_0 = \min D,\\
m_{j+1} = \text{ least $m \in D$ such that $m_j \prec  m$,}
\end{gather*}
and let 
\[
H :=\{m_j \setdef j\in\nat\}. 
\]

\item[Case 2] $\mathcal T$ fails. %
In this case, it is easy to pick a subset of $D$ consisting of pairwise $\prec$-incomparable elements.
Define $m_0, m_1, \hdots$ by induction as follows:
\begin{gather*}
m_0 = \min D,\\
m_{j+1} = \text{ least $m$ such that $\big(\forall m' \in D \setminus m+1\big)\; m_j \not\prec m'$,}
\end{gather*}
and again let 
\[
H :=\{m_j \setdef j\in\nat\}. 
\]
\end{description}
Clearly, $H$ as constructed above is arithmetical in $(D,\prec)$: 
The predicate $\mathcal T$ is $\Pi^0_3(D,\prec)$; 
and the construction of sequences in Case 1 and Case 2 are easily seen to be arithmetical in $(D,\prec)$. 
\end{proof}

We now refine the construction of $\setname$ from Lemma~\ref{l.D.ZF}, paying closer attention to definability.

Given $f \in \cofperm$, we already know $\setname_2(f)$ is arithmetical in $f$.
By the previous lemma, since $\porange{}$ is recursive,
we can find an infinite set $\setname_3(f) \subseteq \setname_2(f)$ which is uniformly arithmetical in $f$ and such that 
$f[\setname_3(f)]$ is $\porange{f}$-homogeneous. In other words, we can choose the map
$\setname_3 \colon \cofperm \to \powerset(\nat)$
to be arithmetical.
Repeating the same argument, 
we can find an arithmetical map $\setname_4 \colon \cofperm \to \powerset(\nat)$ %
such that $\setname_4(f)\subseteq \setname_3(f)$ is infinite and $f[\setname_4(f)]$ is $\pow{f}$-homogeneous.

The predicate ``$h_f$ exists''---that is, \eqref{e.h_f}---holds of $f$ \emph{if and only if}
$f[\setname_4(f)]$ is totally ordered by $\pow{f}$. 
Thus the predicate ``$h_f$ exists'' is obviously arithmetical in $f$.
An analogous argument shows the predicate  ``$w_f$ exists'' to be arithmetical in $f$.

We now verify that the definition of the ``semaphore'' is also arithmetical.
Let us suppose for the moment that $h_f$ and $w_f$ exist.

The relation $h_f(k) =l$ 
is arithmetical in $f$ since 
\[
h_f(k) =l \iff \big(\exists m \in \setname_3(f)\big) (\exists \bar h \in {}^{<\omega}\nat) \;
f(m) \in I_{\#(\bar h)} \land \bar h(k) =l
\]
Similarly, the relation $r^\infty_n(w_f) = \bar w$ 
is arithmetical in $f$ since 
\begin{equation*}
r^\infty_n(w_f) = \bar w \iff (\exists m_0 \in \nat) \big(\forall m \in \setname_4(f)\setminus m_0\big) \; 
 (r_n \circ \codeword)\big(m, h_f(m)\big) = \bar w
\end{equation*}
and because $h_f$ is arithmetical in $f$.
Now a glance at the definition of $\bar y^{h_f, w_f}$ suffices to see that
this sequence is arithmetical in 
$
\big(\setname_2(f), \setname_4(f), w_f, h_f, f\big)
$
Since these are are all arithmetical in $f$, so is $\bar y^{h_f, w_f}$ .
We conclude that
$\setname_5(f)$ can be constructed in an arithmetical-in-$f$ manner.

We thus have constructed a map $\setname_5 \colon \cofperm \to \powerset(\nat)$ as in Lemma~\ref{l.D.ZF} 
but which furthermore is arithmetical. 
We now arrange that there is a predicate $\lambda$ as in \eqref{e.bar.lambda} satisfying the requirements of Proposition~\ref{p.borel}.
Repeating the argument from the beginning of the previous paragraph one last time, 
find an arithmetical map $\setname_6 \colon \cofperm \to \powerset(\nat)$ %
such that $\setname_6(f) \subseteq \setname_5(f)$ is infinite and $\pow{f}$-homogeneous.
Finally, for any $f \in \cofperm$ define
\[
\setname(f) = \setname_6(f).
\]

\begin{lemma}
With this choice of map $\setname\colon \cofperm\to\powerset(\nat)$, the following are equivalent:
\begin{enumerate}
\item There is
$\vec x \in \mathbb{F}\left({}^{\nat} 2\right)$ such that $f \res \set f = \cofmap(\vec x) \res \set f$ 
\item $\big(\forall n,n' \in \set f\big)\; n \po{f} n'$, 
\item $\kappa_{\setname}(f)$, 
i.e., there is
$X \in \set{f}^{[\infty]}$ and $\vec x \in \mathbb{F}\left({}^{\nat} 2\right)$ such that $f \res X = \cofmap(\vec x) \res X$, 
\item $\big(\exists n,n' \in \set f\big)\; n \po{f} n'$, 
\item $\mathcal T\big(\setname(f),\po f\big)$ fails
\end{enumerate}
\end{lemma}
\begin{proof}
Noting that either $\set f$ is pairwise $\po f$-comparable or pairwise $\po f$-incomparable, and that the second possibility holds if and only if $\mathcal T\big(\setname(f),\po f\big)$ holds, the above equivalences are obvious by the definition of $\po f$.
\end{proof}
Thus, letting
\begin{equation}\label{e.lambda}
\lambda(D,f) :\iff (\forall n,n' \in D)\; n \po{f} n',
\end{equation}
all the requirements of Proposition~\ref{p.borel} hold.
The reader may find it helpful to note that alternatively, letting $\lambda(D,f) :\iff (\exists n,n' \in D)\; n \po{f} n'$
would achieve the same goal (this formula being equivalent in the relevant case, i.e., when $D=\set f$, by the previous lemma).

\medskip

By the previous lemma we have constructed a map satisfying all the requirements of Proposition~\ref{p.borel}:
\begin{corollary}
The map 
\begin{gather*}
\setname\colon S_\infty \to \powerset(\nat),\\
f \mapsto \set{f}
\end{gather*}
constructed above satisfies all the requirements of Proposition~\ref{p.borel}.
\end{corollary}
\begin{proof}
Requirements \ref{e.D.first}--\ref{e.D.superspaced} hold because $\set f \subseteq \adsetname_5(f)$ and by the arguments from the previous section. 
By the previous lemma we moreover have $(\forall f \in S_\infty)\; \lambda(\set f,f) \iff \kappa_{\setname}(f)$. 
Also,  %
$\lambda(D,f)$ is $\Pi^0_1$, in particular, it is $\Delta^1_1$.
Finally, by construction $f \mapsto \set f$ is analytical, in particular it is $\Delta^1_1$.
\end{proof}

It is, in fact, not hard to adapt the construction given above in %
this section so that the resulting mcg $\mcg$ is arithmetical. That is, it can be given a definition in second order arithmetical by a formula involving only finitely many quantifiers over natural numbers.\footnote{Again, in fact two quantifiers over $\nat$ suffice---see Theorem~\ref{t.F.sigma} below (without proof).}
\begin{proposition}\label{p.arithm.MCG}
There is an finite level Borel (in fact, arithmetical) MCG which, moreover, is isomorphic to the group $\mathbb{F}\big({}^{\nat} 2\big)$.
\end{proposition}

\begin{remark}\label{r.even}
In the following proof---see \eqref{e.phi.can} below---the reader will finally see why in the definition of $\setname_1(f)$ in \eqref{e.up} on p.~\pageref{e.up}, we made sure that $m_n \notin \setname_1(f)$.
This is convenient since, in notation used below in the proof, it allows us to easily recover $\codeword^h$ from $h \in \mcg$. 
\end{remark}
In this proof we shall finally make use of the fact that $\la G_n \setdef n\in\nat\ra$, $\la c_n \setdef n\in\nat\ra$, $\la I_n \setdef n\in\nat\ra$, and $\cofmap$,  from Section~\ref{s.groundwork} are arithmetically definable (in fact, they are effectively computable).
\begin{proof}
Assume in addition to the requirements stated in Proposition~\ref{p.bp}, that $\ran(\chi)$ is closed and $\chi$ is continuous 
and that, moreover, its graph is $\Pi^0_1$ (or at least, $\ran(\chi)$ and the graph of $\chi$ are arithmetical). 
We have already given a suggestion for an adequate function $\chi$ so that all of the above is true at the beginning of Section~\ref{s.scenic}, just after \eqref{e.chi}.
Finally (recalling that $m_n = \min(I_n)$) we assume that $\set{f}\cap \{m_n \setdef n \in \nat\} = \emptyset$ for each $f\in S_\infty$, as indeed does hold for the map $\setname$ we have constructed above.

Again, we define a MCG $\mcg$ as in \eqref{e.mcg.bp} but with $\kappa$ replaced by $\lambda$ as defined in \eqref{e.lambda}, and of course, with the map $\setname$ as defined on page~\pageref{page.D.2}:
\begin{multline}\label{e.Borel.to.arithm}
h\in\mcg \stackrel{\text{def}}{\iff} (\exists l \in \nat)(\exists g_0,\hdots, g_l \in S_\infty)\big(\exists i_0,\hdots, i_l \in \{1,-1\}\big) \;
h = (g_l)^{i_l} \hdots (g_0)^{i_0} \land  \\
(\forall i \leq l)\; \bigg\{ g_i \in \infG_0 \setminus \ran(\xi)
\lor\\
(\exists f_i \in S_\infty)
(\forall i \leq l)\; 
\Big[\big(
\lambda(\set{f_i},f_i) \land g_i = \cofmapext(f_i) 
\big) 
\lor \\
\big(
\neg\lambda(\set{f_i},f_i) \land g_i = \cofmapext(f_i) \glue{\set{f_i}} f_i 
\big)\Big] \bigg\}.
\end{multline} 
Just as in Proposition~\ref{p.borel}, since $\mcg$ satisfies all the requirements of Proposition~\ref{p.bp} it is a MCG.
We now demonstrate how to find an arithmetical definition of this group $\mcg$.
The idea is that a witness to every quantifier in \eqref{e.Borel.to.arithm}, if such a witness exists at all, is definable from $h$ by an arithmetical relation. Thus, all second-order quantifiers can be eliminated from \eqref{e.Borel.to.arithm}.

It may help at this point to slightly change perspectives regarding the construction of $\mcg_0$ again and recall
the surjective group homomorphism
\[
\maxmap\colon \mathbb{F}\big({}^{\nat} 2\big) \to \mcg.
\]
defined in \eqref{e.maxmap}.
For the readers convenience, we rephrase the definition: For $x \in {}^{\nat}2$, 
\[
\maxmap(x) := \begin{cases}
	\cofmap(x) & \text{if $x \notin \ran(\chi),$}\\
	\cofmap(x)  & \text{if $x \in \ran(\chi)$ and $\lambda(\set{f},f)$, where $f:= \chi^{-1}(x)$}\\
	\cofmapext(f) \glue{\set{f}} f  & \text{if $x \in \ran(\chi)$ and $\neg\lambda(\set{f},f)$, where $f:= \chi^{-1}(x)$.}
\end{cases}
\]
Not that in the second line, $\cofmap(x) =\cofmapext(f)$.
We stress again that \eqref{e.maxmap.C_0} holds, that is, 
\[
\mcg = \ran(\maxmap).
\]
In fact, $\maxmap$ is also injective, as will be seen in the remainder of this proof.

\medskip

Let us make some further definitions:
Let us say \emph{$h$ is a candidate} if and only if $(\exists n_0 \in \nat)\; \phi_{\text{can}}(h,n_0)$, where 
\begin{equation}\label{e.phi.can}
\phi_{\text{can}}(h,n_0) \stackrel{\text{def}}{\iff} (\forall n,n' \in \nat \text{ s.t.\ } n,n' \geq n_0) \; 
\Big[ n < n' \Rightarrow  
m_n \po{h} m_{n'}
\Big].
\end{equation}
Clearly, $h$ is a candidate %
if and only if there is $w \in \mathbb{F}\big({}^{\nat} 2\big)$ such that  
\begin{equation}\label{e.agrees}
(\exists n_0 \in \nat)\; \text{$\cofmap(w)$ agrees with $h$ on $\{m_n \setdef n \in \nat\,\land\, n \geq n_0 \}$.} 
\end{equation}
Here we use that $\{m_n \setdef n \in \nat\,\land\, n \geq n_0 \}$ is never affected by surgery; see Remark~\ref{r.even} above.

Whenever $h \in S_\infty$ is a candidate, let\footnote{This is very different from $w^f$ in the proof of Lemma~\ref{l.D.ZF}. The slogan is, $\codeword^h$ ``codes'' $h$, while $w_f$ ``catches'' $h$ in the context of said proof.} 
\begin{gather*}
\codeword^h := \text{the unique word $w \in \mathbb{F}\big({}^{\nat} 2\big)$ satisfying \eqref{e.agrees},}\\
l(h) := \lh\big(\codeword^h\big),
\end{gather*}
and find 
$x^h_0,\hdots, x^h_l \in {}^{\nat} 2$ and $i^h_0,\hdots, i^h_l \in \{1,-1\}$ such that
\[
\codeword^h = (x^h_l)^{i^h_l} \hdots (x^h_0)^{i^h_0}
\]
where $l = l(h)$.
Moreover, for each $i \leq l^h$, if it should be the case that
$x^h_i \in \ran(\chi)$, we define
\[
f^h_i := \chi^{-1} (x^h_i)
\]
Finally, define $g^h_i$ to be $\maxmap(x^h_i)$, that is,
\[
g^h_i = \begin{cases}
		\chi(x^h_i) &\text{if $x^h_i \notin \ran(\chi)$,}\\
		\cofmapext(f^h_i) \big(= \chi(x^h_i)\big)  &\text{if $x^h_i \in \ran(\chi)$ and $\lambda(\set{f^h_i},f^h_i)$,}\\
		\cofmapext(f^h_i) \glue{\set{f^h_i}} f^h_i  &\text{if $x^h_i \in \ran(\chi)$ and $\neg\lambda(\set{f^h_i},f^h_i)$.}\\
\end{cases}
\]
With these definitions, it is straightforward to verify that
\begin{equation}\label{e.h=...}
h \in \mcg \iff \text{$h$ is a candidate and } h = (g^h_{l(h)})^{i^h_{l(h)}} \hdots (g^h_0)^{i^h_0}.
\end{equation}
It remains to verify that this is an arithmetical property of $h$.
While this is almost immediate from the construction, we give some details for the convenience of the reader.

\begin{claim}\label{c.l}
The relation $\phi_{\textup{can}}$ on ${}^{\nat}\nat \times \nat$ is arithmetical,
and $\codeword^h$ is uniformly arithmetical in $h$.
Moreover, there are arithmetical relations $L$ on ${}^{\nat}\nat \times \nat$ and $I$ on 
${}^{\nat}\nat \times \nat^2$ 
such that 
\begin{align*}
L(h,l) & \iff  \text{$h$ is a candidate and } l = l(h), \\
I(h,j,i) &\iff   (\exists l)\;  L(j,l)\land j<l \land i^h_j=i.
\end{align*}
\end{claim}
\begin{proof}[Proof of claim]
Since $\vec I$, $c_n$, and $G_n$ are computable from $n$, the set 
\[
\{(n,w)  \setdef n\in \nat \;\land\;  w = \diffword\big(m_n, h(m_n)\big)\}
\]
is computable in $h$. 
Thus, also $\pow{h}$ is $\Delta^0_1$ in $h$.
A glance at \eqref{e.phi.can} shows that $\phi_{\text{can}}$
is $\Pi^0_1$ in $h$.
Finally, that
$\codeword^h$ is uniformly $\Sigma^0_2$ in $h$ follows from the fact that %
\[
r_n(\codeword^h) = w \iff \phi_{\text{can}}(n,h) \;\land\; c_n(w) = \diffword\big(m_n, h(m_n)\big).
\]

The second part of the claim follows. Alternatively, take $L(h,l)$ to be  %
\begin{equation*}
 (\exists n\in\nat)\; \Big[\phi_{\text{can}}(h,n) \;\land\;  l = \lh\big(\diffword(m_n, h(m_n))\big)\Big].
\end{equation*}
This  is arithmetical (even $\Sigma^0_2$) in $h$ for the same reasons as cited in the previous paragraph. 
Similarly for $I$.
\renewcommand{\qedsymbol}{{\tiny Claim \ref{c.l}.} $\Box$}
\end{proof}

\begin{claim}\label{c.R}
There are arithmetical relations $R_x, R_f,$ and $R_g$ on ${}^{\nat}\nat \times \nat^3$ such that 
for any candidate $h \in S_\infty$ and $j \leq l(h)$, 
\begin{align*}
R_x(h,j,m,n) & \iff  x^h_j(m)=n, \\
R_f(h,j,m,n) & \iff  x^h_j \in \ran(\chi) \;\land\; f^h_j(m)=n, \\
R_g(h,j,m,n) & \iff  g^h_j(m)=n.
\end{align*}
\end{claim}
\begin{proof}[Proof of claim]
The first equivalence follows from the previous claim. 
Alternatively, one can easily verify that $R_x(h,j,m,n)$ is equivalent to %
\begin{equation*}
 (\exists n'\in\nat)(\exists i \in \{-1,1\}\; \left[\phi_{\text{can}}(h,n') \;\land\;   c_{n'}\Big(\diffword\big(m_{n'}, h(m_{n'})\big)\Big)_j = x^{i} \;\land\; x(m) = n\right].
\end{equation*}
That $R_f$ is also arithmetical follows from our assumption (at the beginning of the proof of Proposition~\ref{p.arithm.MCG}) that $\chi$ is effectively continuous, injective, and has arithmetical (even closed) range.
That $R_g$ is arithmetical follows from the definition of surgery, from the fact that $\set{f}$ is arithmetical in $f$, from the fact that $\lambda(D,f)$ is arithmetical in $D$ and $f$, and from the fact that arithmetical relations are closed under substitutions. The remaining details are left to the reader.
\renewcommand{\qedsymbol}{{\tiny Claim \ref{c.R}.} $\Box$}
\end{proof}

\begin{claim}\label{c.arithm}
The unary relation $R \subseteq {}^{\nat}\nat$ defined by
\[
R(h) \stackrel{\textup{def}}{\iff} \big[\text{$h$ is a candidate and }h = (g^h_{l(h)})^{i^h_{l(h)}} \hdots (g^h_0)^{i^h_0}\big]
\]
is arithmetical.
\end{claim}
\begin{proof}[Proof of claim]
Clearly, $R(h)$ is equivalent to the conjunction of $(\exists n_0\in\nat)\;\phi_{\text{can}}(h, n_0)$ and  
\[
 \big(\exists \vec n \in {}^{l(h)+1}\nat\big)\;
 \vec n(0) = m \;\land\; 
 \vec n\big(l(h)+1\big)=n \;\land\; 
\big(\forall j \leq l(h)\big)\; \vec n(j+1) = \Big(g^h_j\Big)^{i^h_j}\big(\vec n(j)\big).
\]
This is arithmetical by standard arguments, and by substituting the relations $L$, $I$, and $R_g$ from the previous claims.
\renewcommand{\qedsymbol}{{\tiny Claim \ref{c.arithm}.} $\Box$}
\end{proof}
This completes the proof that the right hand side of \eqref{e.h=...}, and hence the MCG defined by \eqref{e.Borel.to.arithm}, is arithmetical.
\end{proof}
\begin{corollary}
Theorem~\ref{t.arithmetical}, a.k.a., Theorem~\ref{t.arithmetic} hold.
\end{corollary}

\medskip

In fact, as we have claimed in the introduction, Theorem~\ref{t.F.sigma} below holds, i.e., 
there is a MCG which is generated by a closed (even $\Pi^0_1$) subset of $S_\infty$.

\begin{remark}\label{r.comparison}
We take a moment to give an incomplete list of differences between the proofs in this paper and the earlier proof by Horowitz and Shelah in \cite{mcg-borel}. There may be further differences that I am not aware of.
The main idea of the strategy sketched at the beginning of Section~\ref{s.bp} is doubtlessly due to Horowitz and Shelah, as is the definition of surgery.
The construction of the map $\cofmap$ differs somewhat from theirs;
the definitions of $\set{f}$ in every section seem to me different as well, and the corresponding construction in \cite{mcg-borel} is, I believe, substantially more complex.
Moreover, our use of the formulas $\kappa_{\setname}$ and especially $\lambda$ differs from the approach in \cite{mcg-borel}; they have a similar case distinction, but their version relies heavily on details of the proof of the Infinite Ramsey Theorem.
Finally, we find explicit conditions on $\setname$ and $\cofmap$, as stated in the present paper in several propositions, clarifying.
With these, we find it easy to arrive at an arithmetical group (the group in \cite{mcg-borel} may well also be arithmetical). 
\end{remark}

\section{The open question}

It was shown in \cite{kastermans-complexity} that no $K_\sigma$ (i.e., countable union of compact sets) subgroup of $S_\infty$ can be maximal cofinitary.

To the following longstanding question, we still do not know the answer: 
\begin{question}\label{q}
Can a closed, or even a $\Pi^0_1$, subgroup of $S_\infty$ be maximal cofinitary?
\end{question}
As has been mentioned several times in this article, with some further tricks one can push the methods in this paper and show the following (cf.\ Remark~\ref{r.simple}):
\begin{theorem}\label{t.F.sigma}
There exists a closed (even $\Pi^0_1$) subset $\mcg_0$ of $S_\infty$ such that the subgroup
$\mcg:=\la \mcg_0 \ra^{S_\infty}$ it generates is maximal cofinitary.
Moreover the MCG $\mcg$ is $F_\sigma$ (even $\Sigma^0_2$).
\end{theorem}
If one restricts attention to \emph{free} maximal cofinitary groups, that is, MCGs which are isomorphic to a free group, this result is optimal.

\bibliography{def-mcg}{}
\bibliographystyle{amsplain}

\end{document}

%% file: preamble.tex
\theoremstyle{plain}
\newtheorem{itheorem}{Theorem}
\newtheorem{theorem}{Theorem}[section]
\newtheorem{lemma}[theorem]{Lemma}
\newtheorem{proposition}[theorem]{Proposition}
\newtheorem{corollary}[theorem]{Corollary}
\newtheorem{claim}[theorem]{Claim}
\newtheorem{subclaim}[theorem]{Subclaim}

\newtheorem{question}[theorem]{Question}

\newtheorem*{claim*}{Claim}
\newtheorem*{subclaim*}{Sublaim}
\newtheorem*{theorem*}{Theorem}
\newtheorem*{lemma*}{Lemma}
\newtheorem*{proposition*}{Proposition}
\newtheorem*{corollary*}{Corollary}

\theoremstyle{definition}

\newtheorem{remark}[theorem]{Remark}

\newtheorem*{definition*}{Definition}
\newtheorem*{example*}{Example}
\newtheorem*{fact*}{Fact}

\theoremstyle{remark}

{\end{minipage}\end{equation}}

\newenvironment{eqpar*}{\begin{equation*}\begin{minipage}{0.8\columnwidth}}%
{\end{minipage}\end{equation*}}

\DeclareMathOperator{\fix}{fix}

\DeclareMathOperator{\powerset}{\mathcal{P}}

\DeclareMathOperator{\nat}{\mathbb{N}}

\providecommand{\Int}{\mathbb{Z}}

\DeclareMathOperator{\dom}{dom}
\DeclareMathOperator{\lh}{lh}

\DeclareMathOperator{\ran}{ran}

\providecommand{\res}{\mathbin{\upharpoonright} }

\providecommand{\la}{ \langle}
\providecommand{\ra}{ \rangle}

\providecommand{\ZF}{\textup{\textsf{ZF}}}

\DeclareMathOperator{\eL}{\mathbf{L}}

\providecommand{\AC}{\textsf{\textup{AC}}}

\providecommand{\setdef}{\;|\;}

\providecommand{\id}{\textup{id}}